\documentclass[11pt]{amsart}
\usepackage{amsmath,amssymb,mathrsfs,color}
\usepackage{graphicx}
\usepackage{epstopdf}
\usepackage{hyperref}
\usepackage{comment}
\usepackage{bbm}
\hypersetup{hypertex=true,
            colorlinks=true,
            linkcolor=blue,
            anchorcolor=blue,
            citecolor=blue}
\newtheorem{thm}{Theorem}[section]
\newtheorem{cor}[thm]{Corollary}
\newtheorem{lem}[thm]{Lemma}
\newtheorem{prop}[thm]{Proposition}

\theoremstyle{definition}
\newtheorem{de}[thm]{Definition}
\theoremstyle{remark}
\newtheorem{rem}[thm]{Remark}

\newtheorem{ass}[thm]{Assumption}
\numberwithin{equation}{section}
\allowdisplaybreaks

\begin{document}

\title[Semimartingale Optimal Transport with Jumps]{Semimartingale Optimal Transport with Jumps: A General Framework and Equivalent Formulations}
\author{Boyi Hou}
\address{B. Hou: School of Mathematics, Dalian University of Technology, Dalian 116024, P. R. China}
\email{boyi.hou@hotmail.com, houboyi@mail.dlut.edu.cn}

\author{Zhenxin Liu}
\address{Z. Liu (Corresponding author): School of Mathematics, Dalian University of Technology, Dalian 116024, P. R. China}
\email{zxliu@dlut.edu.cn}

\thanks{This work is supported by National Key R\&D Program of China (No. 2023YFA1009200), NSFC (Grants 12531009 and 11925102), and Liaoning Revitalization Talents Program (Grant XLYC2202042).}

\date{\today}
\keywords{Dynamic programming; Kantorovich duality; Markovian projection; non-local Fokker–Planck–Kolmogorov equation; semimartingale optimal transport.}
\subjclass[2020]{49Q22, 60G44, 49K45, 60H10, 60J76, 35Q84, 93E20.}

\begin{abstract}
We study a semimartingale optimal transport (SOT) problem where the cost depends on the full differential characteristics, and the minimisation is over all semimartingale laws with given marginals whose absolutely continuous characteristics lie in a prescribed closed convex set.
Under only minimal assumptions of measurability, convexity, and lower semicontinuity on the cost function,
we prove the existence of an optimal plan for SOT and establish a Kantorovich-type duality without time-regularity conditions.
We further prove that SOT admits four equivalent formulations:
(i) a Kantorovich duality formulation,
(ii) a viscosity solution formulation of the Hamilton--Jacobi--Bellman equation,
(iii) a martingale solution formulation via Markovian projection,
(iv) a PDE formulation via weak solutions of a non-local Fokker--Planck--Kolmogorov equation.
The framework simultaneously generalises classical optimal transport,
martingale optimal transport, Schr\"odinger bridge problems, and barycentric weak optimal transport.
\end{abstract}

\maketitle
\section{Introduction}
In the classical optimal transport problem of Monge--Kantorovich, one is given an initial probability distribution $\mu_0$ and a terminal distribution $\mu_1$ on $\mathbb{R}^d$.
An admissible transportation plan is a random vector $(X_0,X_1)$ (equivalently, a joint distribution on $\mathbb{R}^d\times\mathbb{R}^d$) with marginals $\mu_0$ and $\mu_1$; the cost of such a plan is $\mathbb{E}[c(X_0,X_1)]$ for a prescribed cost function $c:\mathbb{R}^d\times\mathbb{R}^d\to[0,+\infty)$, and the Monge--Kantorovich problem consists in minimising this cost over all admissible plans.
Under mild conditions, a duality result was established by Kantorovich, converting the problem into a linear programming formulation.
For the classical duality and further developments of the subject, we refer to Villani~\cite{Villani2009OT}.

When the cost is quadratic, $c(x,y)=|x-y|^2$, Benamou and Brenier~\cite{BB2000} discovered an equivalent dynamical formulation: minimise the action
\[
\int_0^1\!\!\int_{\mathbb{R}^d} |b_t(x)|^2\,\rho_t(\mathrm{d}x)\,\mathrm{d}t
\]
over all time‑dependent families of probability measures $(\rho_t)_{t\in[0,1]}$ and velocity fields $(b_t)_{t\in[0,1]}$ satisfying the continuity equation
\[
\partial_t \rho_t + \text{div}(b_t \rho_t) = 0,\qquad \rho_0 = \mu_0,\ \rho_1 = \mu_1 .
\]
This fluid‑dynamical picture interprets optimal transport as a geodesic flow in the Wasserstein space and admits the probabilistic representation
\[
X_t=X_0+\int_0^t b_s\,\mathrm{d}s,\quad X_0\sim\mu_0,\ X_1\sim\mu_1,
\]
where the drift $b_t$ encodes the velocity field driving the flow.

Semimartingale optimal transport generalises the above probabilistic formulation by replacing the deterministic flow with general semimartingale dynamics, and it ultimately extends to discontinuous semimartingales with jumps, thereby providing a highly flexible dynamical formulation of optimal transport.
The seminal contribution is due to Mikami and Thieullen~\cite{Mikami2006}, who introduced a stochastic mass transportation problem in which $X$ is an $\mathbb{R}^d$-valued continuous semimartingale of the form
\begin{equation}\label{dX=bdt+dW}
X_t = X_0 + \int_0^t b_s\,\mathrm{d}s + W_t
\end{equation}
with $W$ a $d$-dimensional standard Brownian motion under the filtration generated by $X$. The corresponding semimartingale optimal transport problem minimises the cost
\[
V(\mu_0,\mu_1):=\inf \mathbb{E}\Bigl[\int_0^1\ell(t,X_t,b_t)\,\mathrm{d}t\Bigr],
\]
where the infimum is taken over all semimartingales given by~\eqref{dX=bdt+dW} with initial law $\mu_0$ and terminal law $\mu_1$.
Tan and Touzi~\cite{TanTouzi2013AOP} extended this framework to path‑dependent noise by modelling $X$ as
\begin{equation}\label{dX=bdt+sigmadW}
X_t = X_0 + \int_0^t b_s\,\mathrm{d}s + M_t,
\end{equation}
where $M_t$ is the continuous local martingale part of $X$ with quadratic variation $c_t$.
Crucially, their work established a version of the classical Kantorovich duality in this controlled dynamics setting, linking the primal problem to a dual formulation that involves stochastic control and HJB equations.
More recently, Liu and Neufeld~\cite{LiuNeufeld2019TAMS} further generalised the problem to discontinuous semimartingales by providing tightness criteria for their laws: with a truncation function $h$, the process admits the canonical representation
\begin{equation}\label{X=X_0+B+M+J}
X_t = X_0 + \int_0^t b_s\,\mathrm{d}s + M_t^c + M_t^d + J_t,
\end{equation}
where $M_t^c$ is the continuous local martingale part of $X$ with quadratic variation $c_t$, $M_t^d$ is the purely discontinuous local martingale part, and $J_t = \sum_{0 \le s \le t} [\Delta X_s - h(\Delta X_s)]$. In the present paper we adopt \eqref{X=X_0+B+M+J} as our underlying framework and study the following semimartingale optimal transport problem with given initial and terminal marginals
$\mu_0,\mu_1$:
\[
V(\mu_0,\mu_1) = \inf_{\mathbb{P}}\; \mathbb{E}^{\mathbb{P}}\!\left[\int_0^1 L\bigl(t,\omega,b_t^{\mathbb{P}},c_t^{\mathbb{P}},F_t^{\mathbb{P}}\bigr)\,\mathrm{d}t\right],
\]
where the infimum is taken over all semimartingale laws $\mathbb{P}$ whose absolutely continuous differential characteristics $(b^{\mathbb{P}},c^{\mathbb{P}},F^{\mathbb{P}})$ lie $\mathbb{P}\times\mathrm{d}t$-a.s. in a prescribed closed convex set $\Theta\subseteq\mathbb{R}^d\times\mathbb{S}_+^d\times\mathcal{L}$. For each result, we provide the corresponding version in the continuous setting.

The semimartingale optimal transport problem is intimately linked to both the Skorokhod embedding problem (SEP) and martingale optimal transport (MOT)~\cite{BeiglbockPierreFriedrich2013FS, BeiglbockNicolas2016AOP}. As shown by Tan and Touzi~\cite{TanTouzi2013AOP}, the SEP can be reformulated as a semimartingale optimal transport problem with zero drift, thereby translating optimal stopping problems into an optimisation over semimartingale characteristics. A powerful mass transport perspective, pioneered by Beiglböck, Cox, and Huesmann~\cite{BeiglbockCoxHuesmann2017IM}, introduced a monotonicity principle that characterises optimal embeddings and martingale couplings through the geometry of their support sets; this principle systematically recovers all classical optimal embeddings, such as those of Root, Rost, and Azéma--Yor. The rigidity of these monotone sets was further investigated by Huesmann and Stebegg~\cite{HuesmannStebegg2018SPA}, and a complete duality theory for the optimal SEP was subsequently established by Beiglböck, Nutz, and Stebegg~\cite{BeiglbockNutzStebegg2022JEMS}. Further PDE characterisations of barrier‑type solutions for a continuum of marginals were given by Richard, Tan, and Touzi~\cite{RichardTanTouzi2020SIAMJCO}.

Semimartingale optimal transport has also found compelling applications in financial mathematics, especially in model calibration.  Traditional calibration
methods, such as Dupire's formula for local volatility, require a continuum of option prices and are therefore prone to numerical instabilities and interpolation errors. In contrast, the semimartingale optimal transport framework offers a non‑parametric, variational calibration approach that imposes the market quotes of a finite number of European options as discrete constraints within a convex optimisation problem; this methodology projects a chosen reference model onto the set of fully calibrated models while preserving as many desirable features of the reference as possible. We refer to~\cite{GLW2017Matrix, GuoLoeperWang2022MF} for the initial development and to~\cite{JosephLoeperObloj2024QF} for further extensions to stochastic interest rates and path‑dependent derivatives. Empirical evidence strongly supports the presence of jumps in both volatility and returns~\cite{YYEJP2003JF}, and our framework naturally nests such jump‑diffusion specifications; further models incorporating jumps are discussed in~\cite{Bates2022, YYAS2025Bernoulli, YYBC2026Bernoulli}.

Our first main contribution exploits the ``a.s. weak~$L^1$'' lower semicontinuity property of integral functionals. Under minimal regularity assumptions on the cost function ---
joint measurability, convexity in the control variables, and joint lower semicontinuity in both the path and the control --- we establish the lower semicontinuity of the integral
functional, and from this we deduce the existence of an optimal plan and a Kantorovich-type duality. A crucial feature of our approach is that we do
not require any continuity of the cost with respect to the time variable. Such continuity is a standard assumption in the frameworks of~\cite{Mikami2006,
TanTouzi2013AOP, LiuNeufeld2019TAMS, DweikGhoussoubKimPalmer2020Poincare, benamou2024entropicsemimartingaleoptimaltransport}; removing it substantially broadens the scope of semimartingale optimal transport.
In fact, the structural conditions we impose are essentially necessary for the lower semicontinuity of such integral functionals (see Balder~\cite[Theorem~4.12]{Balder1986}),
and therefore our hypotheses are close to optimal.

Our second main contribution is to establish several equivalent formulations of semimartingale optimal transport. More specifically, we obtain the following four equivalent characterisations.
\begin{itemize}
\item \textbf{Kantorovich duality formulation.} This generalises the classical dual formulation of optimal transport: the dual value is expressed as $\sup_{\lambda_1}\bigl[\int \lambda_0^{\lambda_1}\,\mathrm{d}\mu_0 - \int \lambda_1\,\mathrm{d}\mu_1\bigr]$, where $\lambda_1$ runs over all bounded continuous functions and $\lambda_0^{\lambda_1}$ is the initial value of a standard stochastic control problem with terminal cost $\lambda_1$.
\item \textbf{Viscosity solution formulation of the Hamilton--Jacobi--Bellman (HJB) equation.} The function $\lambda_0^{\lambda_1}$ can be characterised as the unique viscosity solution of the associated HJB equation (which is, essentially, a dynamic programming equation) with terminal condition $\lambda_1$.
\item \textbf{Martingale solution formulation via Markovian projection.} By the superposition principle, this formulation lifts the path‑dependent framework to a state‑dependent one. Concretely, for any semimartingale one applies the Markovian projection theorem (also known as the mimicking theorem) to construct a semimartingale with Markovian differential characteristics that matches the one‑dimensional marginal laws, and this process serves as a solution to the martingale problem (equivalently, a weak solution of an SDE).
\item \textbf{PDE formulation.} In the spirit of the original Benamou--Brenier formula and its probabilistic counterpart, the equivalence between martingale solutions and weak solutions of a non‑local Fokker–Planck–Kolmogorov (FPK) equation yields a PDE formulation of semimartingale optimal transport.
\end{itemize}

Finally, we present applications of semimartingale optimal transport theory. As a first application within classical optimal transport theory, we give a new proof of the Benamou--Brenier formula based on our framework. Under mild assumptions, we then establish the equivalence between the Benamou--Brenier formulation and the PDE formulation of martingale optimal transport for continuous-time continuous-path processes. Lastly, we prove the Kantorovich duality for martingale optimal transport with jumps, formulated on the Skorokhod space endowed with the $J_1$-topology.

The remainder of this paper is organised as follows. Section~\ref{2} introduces the semimartingale optimal transport problem and states our first main result. Section~\ref{3} introduces its counterpart on an enlarged space, which linearises the cost functional in the original problem, establishes our first main result, and provides the first equivalent formulation. In Section~\ref{sectmeasele}, by leveraging the stability of the feasible set, we present the standard dynamic programming principle and obtain a viscosity solution characterisation of the associated HJB equation. Section~\ref{5} employs the Markovian projection technique to derive both a martingale solution formulation and a PDE formulation. Finally, Section~\ref{6} discusses applications of these results.

\section{The semimartingale transport problem}\label{2}
\setcounter{equation}{0}
The notation and definitions in this section follow the conventions of~\cite{LiuNeufeld2019TAMS}.
Let $\mathbb{R}^d$ be the $d$-dimensional Euclidean space, endowed with its Borel $\sigma$-field $\mathcal{B}(\mathbb{R}^d)$, and denote by $\mathcal{P}(\mathbb{R}^d)$ the set of all probability measures on $(\mathbb{R}^d, \mathcal{B}(\mathbb{R}^d))$. Let $\Omega := \mathbb{D}([0,1],\mathbb{R}^d)$ be the space of all c\`{a}dl\`{a}g paths $\omega = (\omega_t)_{0\le t\le 1}$ equipped with the usual Skorokhod $J_1$-metric $d_{SK}$ and the corresponding Borel $\sigma$-field $\mathcal{F}$. We write $X=(X_t)_{0\le t\le 1}$ for the canonical process, i.e. $X_t(\omega)=\omega _t$ for every $\omega \in \Omega$, and denote by $\mathbb{F} =(\mathcal{F}_t)_{0\le t \le 1}$ the canonical filtration generated by $X$. The space $\mathcal{P}(\Omega)$ (resp. $\mathcal{M}(\Omega)$) of all probability (resp. finite nonnegative) measures on $(\Omega,\mathcal{F})$, equipped with the topology of weak convergence, is a Polish space.

Let $\mathbb{P} \in \mathcal{P}(\Omega)$ be a probability measure under which the canonical process $X$ is a $\mathbb{P}$-$\mathbb{F}$-semimartingale. That is, there exist c\`{a}dl\`{a}g adapted processes $M$ and $A$, with $M_0=A_0=0$, such that $M$ is a local martingale and $A$ has paths of finite variation $\mathbb{P}$-a.s., and $X = X_0 + M + A$ $\mathbb{P}$-a.s. 

Fix a truncation function $h:\mathbb{R}^d \to \mathbb{R}^d$ which is bounded, continuous and satisfies $h(x)=x$ in a neighbourhood of $0$. We denote by $(B^\mathbb{P}, C^\mathbb{P}, \nu^\mathbb{P})$ the characteristics of the semimartingale $X$, where $B^\mathbb{P}$ is the predictable finite variation part in the canonical decomposition of 
\[
X-\sum_{0\le s\le \cdot } \big(\Delta X_s-h(\Delta X_s)\big),
\]
$C^\mathbb{P}$ is the quadratic covariation of the continuous local martingale part of $X$, and $\nu^\mathbb{P}$ is the compensator of the integer-valued random measure $\mu^X$ associated to the jumps of $X$, with
\[
\mu^X(\omega, \mathrm{d}t, \mathrm{d}x) := \sum_{0\le s \le 1} \mathbf{1}_{\{\Delta X_s(\omega) \neq 0\}} \delta_{(s, \Delta X_s(\omega))}(\mathrm{d}t, \mathrm{d}x).
\] 
See~\cite{JacodShiryaev2003} for more details. We emphasise that $X$ is a $\mathbb{P}$-semimartingale for $\mathbb{F}$ if and only if it has this property for the right-continuous filtration $\mathbb{F}_+$ or the usual augmentation $\mathbb{F}_{+}^{\mathbb{P}}$, and the semimartingale characteristics of $X$ computed with respect to $\mathbb{F}$, $\mathbb{F}_+$ and $\mathbb{F}_{+}^{\mathbb{P}}$ coincide; see~\cite[Proposition~2.2]{NeufeldNutz2014SPA}. This fact allows us to circumvent delicate measurability issues. Consequently, the reader may safely regard all theorems as stated on the original space, while the proofs are carried out under its usual augmentation.

Given a semimartingale law $\mathbb{P}$, we say that $X$ has absolutely continuous characteristics if the triplet processes $(B^\mathbb{P}, C^\mathbb{P}, \nu^\mathbb{P})$ are absolutely continuous in $t$ with respect to the Lebesgue measure, $\mathbb{P}$-a.s. Then we can write $(\mathrm{d}B^\mathbb{P},\mathrm{d}C^\mathbb{P},\mathrm{d}\nu^\mathbb{P}) = (b_t^\mathbb{P}\mathrm{d}t, c_t^\mathbb{P}\mathrm{d}t, F_t^\mathbb{P}\mathrm{d}t)$. The differential characteristics $(b^\mathbb{P}, c^\mathbb{P}, F^\mathbb{P})$ take values in $\mathbb{R}^d \times \mathbb{S}_+^d \times \mathcal{L}$, where $\mathbb{S}_+^d$ denotes the set of all symmetric, nonnegative definite $d \times d$ matrices, equipped with the Frobenius norm and the corresponding Borel $\sigma$-field (thus a Polish space), and
\[
\mathcal{L} := \Bigl\{ F \in \mathcal{M}(\mathbb{R}^d) \Bigm| \int_{\mathbb{R}^d} (|x|^2 \wedge 1) \, F(\mathrm{d}x) < \infty \text{ and } F(\{0\}) = 0 \Bigr\}
\]
denotes the set of L\'evy measures. We endow $\mathcal{L}$ with the topology of weak convergence induced by the bounded continuous functions on $\mathbb{R}^d$ vanishing in a neighbourhood of the origin. Then $\mathcal{L}$ is also Polish; see~\cite[Lemma~2.3]{NeufeldNutz2014SPA}. We then equip $\mathbb{R}^d\times \mathbb{S}_+^d \times \mathcal{L}$ with the corresponding product topology.

Write
$$\mathfrak{P}_{sem}^{ac}:=\big \{ \mathbb{P} \in \mathcal{P}(\Omega) \mid X \text{ has absolutely continuous characteristics} \big \},$$
which is measurable (see~\cite[Theorem~2.5]{NeufeldNutz2014SPA}). Note that absolute continuity does not depend on the choice of the truncation function $h$ (see~\cite[Proposition~2.24]{JacodShiryaev2003}).

For a given measurable subset $\Theta \subseteq \mathbb{R}^d\times \mathbb{S}_+^d \times \mathcal{L}$, we define
$$\mathfrak{P}_{\Theta}:=\big \{ \mathbb{P} \in \mathfrak{P}_{sem}^{ac} \mid \big ( b^\mathbb{P}, c^\mathbb{P}, F^\mathbb{P} \big ) \in \Theta \quad \mathbb{P} \times \mathrm{d}t \text{-a.s.} \big \},$$
the set of all semimartingale laws in $\mathfrak{P}_{sem}^{ac}$ whose differential characteristics take values in $\Theta$ $\mathbb{P} \times \mathrm{d}t$-a.s. It is worth noting that $\mathfrak{P}_{\Theta}$ is measurable whenever $\Theta$ is measurable (see~\cite[Theorem~2.6]{NeufeldNutz2014SPA}).

For a probability measure $\mu_0 \in \mathcal{P}(\mathbb{R}^d)$ we further set
$$\mathfrak{P}_{\Theta}(\mu_0):=\big \{ \mathbb{P} \in \mathfrak{P}_{\Theta} \mid \mathbb{P} \circ X_0^{-1} = \mu_0 \big \},$$
the set of all elements in $\mathfrak{P}_{\Theta}$ with initial distribution $\mu_0$. Similarly, for another probability measure $\mu_1 \in \mathcal{P}(\mathbb{R}^d)$, 
$$\mathfrak{P}_{\Theta}(\mu_0, \mu_1):=\big \{ \mathbb{P} \in \mathfrak{P}_{\Theta} (\mu_0) \mid \mathbb{P} \circ X_1^{-1} = \mu_1 \big \}$$
denotes the set of all elements in $\mathfrak{P}_{\Theta}(\mu_0)$ with terminal distribution $\mu_1$.

Next, we associate to each $\mathbb{P}$ a transportation cost
\begin{equation}\label{JP}
J(\mathbb{P}):=\mathbb{E}^{\mathbb{P}}\Bigl[\int_{0}^{1} L\big(t, X, b_t^{\mathbb{P}}, c_t^{\mathbb{P}}, F_t^{\mathbb{P}}\big) \, \mathrm{d}t \Bigr],
\end{equation}
where $L:[0,1] \times \Omega \times \Theta \to \mathbb{R}^+ $ is a given cost function. 

Our main object of study is the following semimartingale optimal transport problem, given two probability measures $\mu_0, \mu_1 \in \mathcal{P}(\mathbb{R}^d)$:
\begin{align}\label{SOT}
V(\mu_0, \mu_1) := \inf_{\mathbb{P} \in \mathfrak{P}_{\Theta}(\mu_0, \mu_1)} J (\mathbb{P}),
\end{align}
with the convention $\inf \emptyset = +\infty$.

We now introduce the duality theorem corresponding to \eqref{SOT}, which can be regarded as an extension of the classical Kantorovich duality in optimal transport theory.
We begin by stating a set of conditions that guarantee compactness of the admissible set of characteristics.
\begin{ass}\label{B}
A set $\Theta \subseteq \mathbb{R}^d \times \mathbb{S}_+^d \times \mathcal{L}$ satisfies
\begin{equation*}
\sup_{(b,c,F)\in \Theta } \bigl \{ |b|+|c|+\int_{\mathbb{R}^d} (|x|^2 \wedge |x|) \, F(\mathrm{d}x) \bigr \} < \infty.
\end{equation*}
\end{ass}

\begin{ass}\label{J}
A set $\Theta \subseteq \mathbb{R}^d \times \mathbb{S}_+^d \times \mathcal{L}$ satisfies
\begin{equation*}
\lim_{\delta  \downarrow  0} \sup_{(b,c,F)\in \Theta}\int_{\{ | z | \le \delta \} } | z | ^2 \, F(\mathrm{d}z)=0.
\end{equation*}
\end{ass}

\begin{rem}\label{Thetacompact}
If $\Theta$ is closed, convex and satisfies Assumptions~\ref{B} and~\ref{J}, then 
$\mathfrak{P}_{\Theta}(\mu_0,\mu_1)$ is a compact convex set under the weak convergence 
topology. This is a direct consequence of an Arzel{\`a}--Ascoli type result proved in 
\cite[Theorem~2.5]{LiuNeufeld2019TAMS}.
The same compactness property continues to hold when $\mathfrak{P}_{\Theta}(\mu_0,\mu_1)$ 
is replaced by the larger set 
$\mathfrak{P}_{\Theta}(\Gamma_0,\Gamma_1):=\{\mathbb{P}\in \mathfrak{P}_{\Theta}\mid\mathbb{P} \circ X_{0}^{-1} \in \Gamma_0,\,\mathbb{P} \circ X_{1}^{-1} \in \Gamma_1 \}$, 
where $\Gamma_0,\,\Gamma_1\subseteq\mathcal{P}(\mathbb{R}^d)$ are compact. In particular, under 
the stated conditions, the set $\Theta$ itself is compact.
\end{rem}

%We then have the following Arzelà–Ascoli type result~\cite[Theorem~2.5]{LiuNeufeld2019TAMS}:
%\begin{lem}\label{PThetacompact}
%If $\Theta$ is closed, convex and satisfies Assumptions~\ref{B} and~\ref{J}, then $\mathfrak{P}_{\Theta}(\mu_0,\mu_1)$ is a compact convex set under the weak convergence topology.
%\end{lem}
%\begin{rem}
%The result also holds when $\mathfrak{P}_{\Theta}(\mu_0,\mu_1)$ is replaced by $\mathfrak{P}_{\Theta}(\Gamma):=\big \{ \mathbb{P} \in \mathfrak{P}_{\Theta} \mid \mathbb{P} \circ X_0^{-1} \in \Gamma \big \}$, where $\Gamma\subseteq\mathcal{P}(\mathbb{R}^d)$ is compact. In particular, $\Theta$ itself is compact under these conditions.
%\end{rem}

Next, we impose the following condition on the cost function $L$, which is needed for our results.
\begin{ass}\label{L}
The cost function $L:[0,1]\times\Omega\times\Theta\to[0,+\infty)$ is nonnegative, jointly measurable, jointly lower semicontinuous in $(\omega,\theta)$ and convex in $\theta$.
\end{ass}

We now state a lower semicontinuity and convexity result for the semimartingale optimal transport problem. The proof is relegated to Section~\ref{3}.
\begin{thm}\label{Vsemicontinuous}
Let $\Theta\subseteq\mathbb{R}^d\times\mathbb{S}_+^d\times\mathcal{L}$ be closed, convex, and satisfy Assumptions~\ref{B} and~\ref{J}. Moreover, let the cost function $L$ satisfy Assumption~\ref{L}. 
Then the function
\begin{equation}
V: \mathcal{P}(\mathbb{R}^d)\times\mathcal{P}(\mathbb{R}^d) \to [0,+\infty], \quad (\mu_0,\mu_1)\mapsto V(\mu_0,\mu_1)
\end{equation}
is lower semicontinuous and convex.
\end{thm}

Then, we define the dual formulation of \eqref{SOT} by 
\begin{equation}\label{DP}
\mathcal{V}(\mu_0,\mu_1):=\sup_{\lambda_1\in C_b(\mathbb{R}^d)}\Bigl\{ \int_{\mathbb{R}^d}\lambda_0^{\lambda_1}(x)\,\mu_0(\mathrm{d}x)-\int_{\mathbb{R}^d}\lambda_1(x)\,\mu_1(\mathrm{d}x) \Bigr\}, 
\end{equation}
where 
\begin{equation}\label{Rc}
\lambda_0^{\lambda_1}(x):=\inf_{\mathbb{P}\in\mathfrak{P}_{\Theta}(\delta_x)}\mathbb{E}^{\mathbb{P}}\Bigl[\int_{0}^{1} L\big(t, X, b_t^{\mathbb{P}}, c_t^{\mathbb{P}}, F_t^{\mathbb{P}}\big) \, \mathrm{d}t+\lambda_1(X_1) \Bigr].
\end{equation}
Moreover, for any $\mu_0\in\mathcal{P}(\mathbb{R}^d)$, if Assumption~\ref{L} holds, then by~\cite[Lemma~5.4]{LiuNeufeld2019TAMS} the function $\lambda_0^{\lambda_1}$ is measurable with respect to the Borel $\sigma$-field on $\mathbb{R}^d$ completed by $\mu_0$, and
\begin{equation}\label{intRc}
\int_{\mathbb{R}^d}\lambda_0^{\lambda_1}(x)\,\mu_0(\mathrm{d}x)=\inf_{\mathbb{P}\in\mathfrak{P}_{\Theta}(\mu_0)}\mathbb{E}^{\mathbb{P}}\Bigl[\int_{0}^{1} L\big(t, X, b_t^{\mathbb{P}}, c_t^{\mathbb{P}}, F_t^{\mathbb{P}}\big) \, \mathrm{d}t+\lambda_1(X_1) \Bigr].
\end{equation}
The main duality theorem reads as follows. A direct proof can be given by combining Theorem~\ref{Vsemicontinuous} with classical convex analysis, essentially following the lines of~\cite[Theorem~2.16]{LiuNeufeld2019TAMS} (which in turn builds on~\cite[Theorem~3.6]{TanTouzi2013AOP} and~\cite[Theorem~2.1]{Mikami2006}). 
%For the sake of completeness, we relegate the proof to the Supplementary Material.
For completeness, we present the proof.
\begin{thm}\label{D}
Let $\Theta\subseteq\mathbb{R}^d\times\mathbb{S}_+^d\times\mathcal{L}$ be closed, convex, and satisfy Assumptions~\ref{B} and~\ref{J}. Moreover, let the cost function $L$ satisfy Assumption~\ref{L}. Then we have
\begin{equation}
V(\mu_0,\mu_1)=\mathcal{V}(\mu_0,\mu_1)\quad \text{for all } \mu_0,\mu_1\in\mathcal{P}(\mathbb{R}^d),
\end{equation}
and the infimum in \eqref{SOT} is attained by some $\mathbb{P}^{\ast}\in\mathfrak{P}_{\Theta}(\mu_0,\mu_1)$ whenever $V(\mu_0,\mu_1)<\infty$.
\end{thm}
\begin{proof}
Recall from Remark~\ref{Thetacompact} and Theorem~\ref{Vsemicontinuous} that a minimiser of the semimartingale optimal transport problem \eqref{SOT} exists whenever $V(\mu_0,\mu_1)<+\infty$.

To obtain the duality, first consider the case where $V(\mu_0,\cdot)\equiv+\infty$. Then $J(\mathbb{P})=+\infty$ for all $\mathbb{P}\in\mathfrak{P}_{\Theta}(\mu_0)$, and by the definition of the dual problem together with \eqref{intRc} we obtain $\mathcal{V}(\mu_0,\mu_1)=+\infty$ for every $\mu_1$; thus the duality holds trivially.

Assume now that $V(\mu_0,\cdot)$ is not identically infinite. Let $\mathfrak{M}_{f,s}(\mathbb{R}^d)$ be the space of finite signed measures on $\mathbb{R}^d$, equipped with the coarsest topology making the maps $\mu\mapsto\int\phi\,\mathrm{d}\mu$ continuous for all $\phi\in C_b(\mathbb{R}^d)$. The relative topology on the subset $\mathcal{P}(\mathbb{R}^d)$ then coincides with its usual weak topology.
Extend $V(\mu_0,\cdot)$ to the whole space $\mathfrak{M}_{f,s}(\mathbb{R}^d)$ by setting $V(\mu_0,\mu):=+\infty$ for every $\mu\in\mathfrak{M}_{f,s}(\mathbb{R}^d)\setminus\mathcal{P}(\mathbb{R}^d)$. This extension clearly preserves lower semicontinuity and convexity.
Recall that the topological dual $\mathfrak{M}_{f,s}(\mathbb{R}^d)^*$ can be identified with $C_b(\mathbb{R}^d)$ via the natural pairing (see~\cite[Lemma~3.2.3]{DeuschelStroock1989}). Define the Legendre transform $g:C_b(\mathbb{R}^d)\to(-\infty,\infty]$ of the extended functional by
\[
g(\lambda_1):=\sup_{\mu\in\mathfrak{M}_{f,s}(\mathbb{R}^d)}\Bigl\{\int_{\mathbb{R}^d}\lambda_1(x)\,\mu(\mathrm{d}x)-V(\mu_0,\mu)\Bigr\},\quad \lambda_1\in C_b(\mathbb{R}^d).
\]
Applying the classical convex duality theorem (see~\cite[Theorem~2.2.15, p.~55]{DeuschelStroock1989}) yields the bidual representation
\begin{equation}
V(\mu_0,\mu_1)=\sup_{\lambda_1\in C_b(\mathbb{R}^d)}\Bigl\{\int_{\mathbb{R}^d}\lambda_1(x)\,\mu_1(\mathrm{d}x)-g(\lambda_1)\Bigr\},\quad \mu_1\in\mathcal{P}(\mathbb{R}^d).
\end{equation}
To identify $g$, we evaluate $g(-\lambda_1)$:
\begin{align*}
g(-\lambda_1)
&=\sup_{\mu\in\mathfrak{M}_{f,s}(\mathbb{R}^d)}\Bigl\{\int_{\mathbb{R}^d}-\lambda_1(x)\,\mu(\mathrm{d}x)-V(\mu_0,\mu)\Bigr\}\\
&=\sup_{\mu\in\mathcal{P}(\mathbb{R}^d)}\Bigl\{\int_{\mathbb{R}^d}-\lambda_1(x)\,\mu(\mathrm{d}x)-V(\mu_0,\mu)\Bigr\}\\
&=-\inf_{\mu\in\mathcal{P}(\mathbb{R}^d)}\Bigl\{\int_{\mathbb{R}^d}\lambda_1(x)\,\mu(\mathrm{d}x)+V(\mu_0,\mu)\Bigr\}.
\end{align*}
Using the definition of \eqref{SOT} and \eqref{intRc}, the infimum on the right‑hand side can be rewritten as
\begin{align*}
\inf_{\mu\in\mathcal{P}(\mathbb{R}^d)}\Bigl\{\int_{\mathbb{R}^d}\lambda_1(x)\,\mu(\mathrm{d}x)+V(\mu_0,\mu)\Bigr\}
&=\inf_{\mu\in\mathcal{P}(\mathbb{R}^d)}\inf_{\mathbb{P}\in\mathfrak{P}_{\Theta}(\mu_0,\mu)}\bigl\{\mathbb{E}^{\mathbb{P}}[\lambda_1(X_1)]+J(\mathbb{P})\bigr\}\\
&=\inf_{\mathbb{P}\in\mathfrak{P}_{\Theta}(\mu_0)}\bigl\{\mathbb{E}^{\mathbb{P}}[\lambda_1(X_1)]+J(\mathbb{P})\bigr\}\\
&=\int_{\mathbb{R}^d}\lambda_0^{\lambda_1}(x)\,\mu_0(\mathrm{d}x).
\end{align*}
Consequently, $g(-\lambda_1)=-\int_{\mathbb{R}^d}\lambda_0^{\lambda_1}(x)\,\mu_0(\mathrm{d}x)$. Hence we obtain
\begin{align*}
V(\mu_0,\mu_1)
&=\sup_{\lambda_1\in C_b(\mathbb{R}^d)}\Bigl\{\int_{\mathbb{R}^d}-\lambda_1(x)\,\mu_1(\mathrm{d}x)-g(-\lambda_1) \Bigr\}\\
&=\sup_{\lambda_1\in C_b(\mathbb{R}^d)}\Bigl\{-\int_{\mathbb{R}^d}\lambda_1(x)\,\mu_1(\mathrm{d}x)+\int_{\mathbb{R}^d}\lambda_0^{\lambda_1}(x)\,\mu_0(\mathrm{d}x) \Bigr\}\\
&=\mathcal{V}(\mu_0,\mu_1),
\end{align*}
which completes the proof.
\end{proof}

For continuous semimartingales ($F\equiv 0$), the boundedness condition on $\Theta$ can be dropped by imposing coercivity on $L$:
\begin{ass}\label{uL}
$L$ is coercive in the sense that
\begin{equation}
\lim_{\max\{|b|,|c|\}\to+\infty}\inf_{(t,\omega)}\frac{L(t,\omega,b,c,0)}{|b|+|c|}=+\infty.
\end{equation}
\end{ass}
Under this coercivity condition, we have the following counterpart of Theorem~\ref{Vsemicontinuous} and~\ref{D} for continuous semimartingales.
\begin{prop}\label{F=0SOT}
Let $\Theta\subseteq\mathbb{R}^d\times\mathbb{S}_+^d\times\{0\}$ be closed and convex,
and let $L$ satisfy Assumptions~\ref{L} and~\ref{uL}. Then $V$ is lower semicontinuous
and convex, and the conclusions of Theorem~\ref{D} hold.
\end{prop}
\begin{rem}
The proof of lower semicontinuity and convexity follows from Lemma~8.4 of \cite{BVPammer2022Bernoulli}, replacing the Ioffe theorem~\cite[Theorem 1]{Ioffe19771} ---
which, in our view, is implicitly employed there --- with Balder's theorem (\cite[Theorem~4.9]{Balder1986}; see also \cite[Theorem~3.2]{Balder1984}). The proof of the remaining conclusions of Proposition~\ref{F=0SOT} is then identical to that of Theorem~\ref{D} itself.
\end{rem}

\section{Lower semicontinuity and convexity}\label{3}
This section is devoted to the proof of Theorem~\ref{Vsemicontinuous}. The main difficulty is that the integral functional $J(\mathbb{P})$ defined in \eqref{JP} depends nonlinearly on the semimartingale characteristics of $X$, which makes a direct application of standard convex analytic tools delicate. To overcome this, we lift the original problem to an enlarged space on which the relevant characteristics appear as pathwise derivatives, thereby linearising the dependence of the integral functional on the underlying probability measure.  Once this linearisation is in place, the lower semicontinuity and convexity of the (extended) integral functional become amenable to the general theory of integral functionals due to Balder~\cite{Balder1986}. The main steps of this section are as follows. In Subsection~\ref{enlarged} we construct the enlarged space, establish the precise correspondence between the original and the enlarged problems. In Subsection~\ref{proof_Vsemicontinuous} we prove the key lower semicontinuity result for the extended integral functional and then combine it with the correspondence lemma to conclude Theorem~\ref{Vsemicontinuous}.

\subsection{An enlarged space}\label{enlarged}
In this subsection, we introduce an enlarged space.

Let $\mathcal{C}^+(\mathbb{R}^d):=\{g_i \mid i \in \mathbb{N}\}$ be a countable family of bounded continuous functions on $\mathbb{R}^d$ vanishing in a neighbourhood of the origin with the following properties:
\begin{itemize}
\item For each $i \in \mathbb{N}$, let $g_{2i-1}:=((q_i\cdot|x|-1)\vee 0)\wedge 1$, where $q_i$ is the $i$-th positive rational number in Cantor's diagonal enumeration.
\item For each $i \in \mathbb{N}$, let $g_{2i}$ be a continuous function on $\mathbb{R}^d$  satisfying for  all $x \in \mathbb{R}^d$ that $0\leq g_{2i}(x)\leq |x|^2 \wedge 1 $ and that
\[
g_{2i}(x) =
\begin{cases}
0,       & \quad \text{if } |x| \leq \frac{1}{2i},\\
|x|^2 \wedge 1,  & \quad \text{if } |x| > \frac{1}{i}.
\end{cases}
\]
\item It is a law-determining class for L\'evy measures, i.e., for any two L\'evy measures $F, F^\prime \in \mathcal{L}$, if $\int_{\mathbb{R}^d} g_i(z) \, F(\mathrm{d}z) = \int_{\mathbb{R}^d} g_i(z)\,F^\prime(\mathrm{d}z)$ for all $i\in \mathbb{N}$, then $F = F^\prime$.
\item It is a convergence-determining class for the weak convergence induced by the bounded continuous functions on $\mathbb{R}^d$ vanishing in a neighbourhood of the origin, i.e., given any sequence $(F_n)_{n \in \mathbb{N}} \subseteq \mathcal{L}$ and $F \in \mathcal{L}$, if $\lim_{n \to \infty} \int_{\mathbb{R}^d} g_i(z)\,F_n(\mathrm{d}z) = \int_{\mathbb{R}^d} g_i(z)\,F(\mathrm{d}z)$ for all $i\in \mathbb{N}$, then $\lim_{n \to \infty} \int_{\mathbb{R}^d} g(z)\,F_n(\mathrm{d}z) = \int_{\mathbb{R}^d} g(z)\,F(\mathrm{d}z)$ for all bounded continuous functions on $\mathbb{R}^d$ vanishing in a neighbourhood of the origin.
\end{itemize}
We refer to~\cite[II.2.20]{JacodShiryaev2003} together with~\cite[VII.2.7 \& 2.8]{JacodShiryaev2003} for the existence of such a class $\mathcal{C}^+(\mathbb{R}^d)$.

For each $n\in\mathbb{N}$, write $\Omega_n=\mathbb{D}([0,1],\mathbb{R}^n)$ endowed with the Skorokhod $J_1$-topology and $\Omega_n^c=\mathbb{C}([0,1],\mathbb{R}^n)$ endowed with the uniform convergence topology (which coincides with the subspace topology of $\Omega_n$). 
We then define the enlarged space
\begin{equation*}
\overline{\Omega} := \Omega_d \times \Omega^c_d \times \Omega^c_d \times \Omega_d \times \Omega_d \times \Omega^c_{d^2}\times \{\Omega^c_1\}^\mathbb{N},
\end{equation*}
equipped with the product topology; this is a Polish space. 
We write
\begin{equation*}
\overline{\mathbb{X}}:=\big( \overline{X},\overline{B},\overline{M}^c,\overline{M}^d,\overline{J}, \overline{C}, \{\overline{V}^i\}_{i\in \mathbb{N}}\big)
\end{equation*}
for the canonical process $\overline{\mathbb{X}}_t(\overline{\omega})=\overline{\omega}(t)$, $\overline{\omega} \in \overline{\Omega}$. We endow $\overline{\Omega}$ with its Borel $\sigma$-field $\overline{\mathcal{F}}$ and denote by $\overline{\mathbb{F}} =(\overline{\mathbb{F}}_t)_{0\leq t \leq 1}$ the (raw) filtration generated by $\overline{\mathbb{X}}$.

Consider a probability measure $\overline{\mathbb{P}}$ on $(\overline{\Omega},\overline{\mathcal{F}})$ under which $\overline{X}$ is a semimartingale with absolutely continuous characteristics, and $\overline{X}$ admits the canonical representation
\[
\overline{X}=\overline{X}_0+\overline{B}+\overline{M}^c+\overline{M}^d+\overline{J}, \quad \overline{\mathbb{P}}\text{-a.s.},
\]
satisfying the following:
\begin{itemize}
  \item $\overline{M}^c_0 = \overline{M}^d_0 = \overline{B}_0 = 0$ $\overline{\mathbb{P}}$-a.s.
  \item $\overline{B}$ has $\overline{\mathbb{P}}$-a.s.\ paths of finite variation.
  \item $\overline{M}^c$ is a continuous $\overline{\mathbb{P}}$-$\overline{\mathbb{F}}$-local martingale with quadratic covariation $\overline{C}$.
  \item $\overline{M}^d$ is a purely discontinuous $\overline{\mathbb{P}}$-$\overline{\mathbb{F}}$-local martingale.
  \item $\overline{J}= \sum_{0\leq s \leq \cdot} \bigl[\Delta \overline{X}_s - h(\Delta \overline{X}_s)\bigr]$ $\overline{\mathbb{P}}$-a.s.
  \item $\overline{\mathbb{P}}$-a.s., each $\overline{V}^i$ coincides with $\int_0^{\cdot}\int_{\mathbb{R}^d} g_i(x)\, \nu^{\overline{\mathbb{P}}}(\mathrm{d}x,\mathrm{d}t)$, 
  where $\nu^{\overline{\mathbb{P}}}(\mathrm{d}x,\mathrm{d}t)$ denotes the $\overline{\mathbb{P}}$-$\overline{\mathbb{F}}$-compensator of the measure $\mu^{\overline{X}}(\mathrm{d}x,\mathrm{d}t)$ associated to the jumps of $\overline{X}$, and $g_i$ denotes the $i$-th element of $\mathcal{C}^+(\mathbb{R}^d)$.
\end{itemize}
We denote by $\overline{\mathfrak{P}}_{sem}^{ac}$ the set of all such probability measures.

Given a set $\Theta \subseteq \mathbb{R}^d \times \mathbb{S}_+^d \times \mathcal{L}$ and two probability measures $\mu_0, \mu_1 \in \mathcal{P}(\mathbb{R}^d)$, we define, consistently with the definitions on the original space, the following sets:
\begin{align*}
\overline{\mathfrak{P}}_{\Theta}      
&:=\bigl\{\overline{\mathbb{P}}\in\overline{\mathfrak{P}}_{sem}^{ac}\mid (\overline{b}^{\overline{\mathbb{P}}}, \overline{c}^{\overline{\mathbb{P}}}, \overline{F}^{\overline{\mathbb{P}}})\in\Theta\quad\overline{\mathbb{P}}\times\mathrm{d}t\text{-a.s.}\bigr\},\\
\overline{\mathfrak{P}}_{\Theta}(\mu_0)&:=\bigl\{\overline{\mathbb{P}} \in \overline{\mathfrak{P}}_{\Theta} \mid \overline{\mathbb{P}} \circ \overline{X}_0^{-1}=\mu_0\bigr\},\\
\overline{\mathfrak{P}}_{\Theta}(\mu_0,\mu_1)&:=\bigl\{\overline{\mathbb{P}}\in\overline{\mathfrak{P}}_{\Theta}(\mu_0)\mid\overline{\mathbb{P}}\circ\overline{X}_1^{-1}=\mu_1\bigr\}.
\end{align*}

To introduce a corresponding cost function on the extended space, let $\{g_i \mid i \in \mathbb{N}\} = \mathcal{C}^+(\mathbb{R}^d)$ and define an additive, positively homogeneous mapping $\varphi : \mathcal{L} \rightarrow \mathbb{R}^{\mathbb{N}}$ by
\begin{equation}\label{eq:map}
\varphi(F) := \Bigl(\int_{\mathbb{R}^d} g_i(z)\, F(\mathrm{d}z)\Bigr)_{i \in \mathbb{N}}.
\end{equation}
Since $\mathcal{C}^+(\mathbb{R}^d)$ is a law-determining class on $\mathcal{L}$, $\varphi$ is injective.
Next, define $\overline{\varphi}: \mathbb{R}^d\times\mathbb{S}_+^d \times \mathcal{L} \to \mathbb{R}^d\times\mathbb{S}_+^d \times\mathbb{R}^{\mathbb{N}}$ by
\[
\overline{\varphi}(b,c,F) := (b,c, \varphi(F)).
\]
Note that $\mathbb{R}^{\mathbb{N}}$ is a Polish space with metric $d(x,y) = \sum_{i=1}^{\infty} \frac{1}{2^i}\,\frac{|x_i - y_i|}{1 + |x_i - y_i|}$.
Clearly, $\overline{\varphi}$ is additive, positively homogeneous, and a bijection onto its image. In the sequel we will mainly use $\overline{\varphi}$ via the following pathwise derivatives. Define the processes
\begin{equation}\label{bcv}
\begin{aligned}
\overline{b}_t&:=\limsup_{n\to \infty}n\bigl(\overline{B}_t-\overline{B}_{(t-\frac{1}{n})\vee 0}\bigr),\\
\overline{c}_t&:=\limsup_{n\to \infty}n\bigl(\overline{C}_t-\overline{C}_{(t-\frac{1}{n})\vee 0}\bigr),\\
\overline{v}^i_t&:=\limsup_{n \to \infty} n \bigl(\overline{V}^i_t -\overline{V}^i_{(t-\frac{1}{n})\vee0}\bigr), \quad t\in [0,1],
\end{aligned}
\end{equation}
for each $i \in \mathbb{N}$.
Set $\overline{v}_t := \{\overline{v}^i_t\}_{i\in\mathbb{N}}$. Whenever $\overline{\mathbb{P}} \in \overline{\mathfrak{P}}_{sem}^{ac}$, we have
\begin{equation}\label{bcF=bcv}
\Bigl(\overline{b}_t^{\overline{\mathbb{P}}}, \overline{c}_t^{\overline{\mathbb{P}}}, \bigl( \int_{\mathbb{R}^d} g_i(z)\,\overline{F}_t^{\overline{\mathbb{P}}}(\mathrm{d}z) \bigr)_{i \in \mathbb{N}} \Bigr)
= \bigl( \overline{b}_t, \overline{c}_t, \overline{v}_t \bigr) \qquad \overline{\mathbb{P}}\times\mathrm{d}t\text{-a.s.}
\end{equation}
Consequently, for any $\Theta \subseteq\mathbb{R}^d\times \mathbb{S}_+^d\times \mathcal{L}$ and $\overline{\mathbb{P}}\in \overline{\mathfrak{P}}_{sem}^{ac}$,
\begin{equation}
(\overline{b}^{\overline{\mathbb{P}}}, \overline{c}^{\overline{\mathbb{P}}}, \overline{F}^{\overline{\mathbb{P}}}) \in \Theta \;\; \overline{\mathbb{P}}\times\mathrm{d}t\text{-a.s.}
\;\Longleftrightarrow\;
(\overline{b}, \overline{c}, \overline{v}) \in \overline{\varphi}(\Theta) \;\; \overline{\mathbb{P}}\times\mathrm{d}t\text{-a.s.}
\end{equation}

Finally, we associate to $\overline{\mathbb{P}}$ a transportation cost. Given a cost function $L$, define $\overline{L}:[0,1]\times\Omega\times\overline{\varphi}(\Theta)\to\mathbb{R}$ by
\begin{equation}\label{deoverl}
\overline{L}\bigl(t, \omega, \overline{\varphi}(b,c,F)\bigr) := L(t, \omega, b, c, F),
\end{equation}
and set
\begin{equation}
\overline{J}(\overline{\mathbb{P}})
= \mathbb{E}^{\overline{\mathbb{P}}}\!
\Bigl[\int_{0}^{1}\overline{L}\bigl(t,\overline{X},\overline{\varphi}
(\overline{b}_t^{\overline{\mathbb{P}}},\overline{c}_t^{\overline{\mathbb{P}}},\overline{F}_t^{\overline{\mathbb{P}}})\bigr)\,\mathrm{d}t\Bigr]
= \mathbb{E}^{\overline{\mathbb{P}}}\!\Bigl[\int_{0}^{1}\overline{L}\bigl(t, \overline{X}, \overline{b}_t,\overline{c}_t,\overline{v}_t\bigr)\,\mathrm{d}t\Bigr].
\end{equation}

We can then define the map $\Psi^{\mathbb{P}}:\Omega\to\overline{\Omega}$ by
\begin{equation}\label{Psi}
\omega\mapsto\bigl(X(\omega), B^{\mathbb{F}^{\mathbb{P}}_{+}}(\omega), M^{c,\mathbb{F}^{\mathbb{P}}_{+}}(\omega), M^{d,\mathbb{F}^{\mathbb{P}}_{+}}(\omega), J^{\mathbb{F}^{\mathbb{P}}_{+}}(\omega), C^{\mathbb{F}^{\mathbb{P}}_{+}}(\omega), \bigl\{V^{i,\mathbb{F}_{+}^{\mathbb{P}}}(\omega)\bigr\}_{i\in\mathbb{N}}\bigr),
\end{equation}
which is measurable with respect to the Borel $\sigma$-field completed by $\mathbb{P}$. The canonical representation of $X$ with respect to $\mathbb{F}^{\mathbb{P}}_+$ guarantees that for every $\omega$ each summand has c\`{a}dl\`{a}g paths or continuous paths, respectively, and not merely $\mathbb{P}$-almost surely, so that $\Psi^{\mathbb{P}}$ is well defined. Since the characteristics of $X$ do not depend on the choice of the filtration, we have that $\mathbb{P} \in \mathfrak{P}_{\Theta}$ implies $\overline{\mathbb{P}} := \mathbb{P} \circ (\Psi^{\mathbb{P}})^{-1} \in \overline{\mathfrak{P}}_{\Theta}$; i.e., $\overline{\mathbb{P}}$ preserves the structure of $\mathbb{P}$.

Under the given transportation cost, there exists a one-to-one correspondence between the original space and the extended space; see~\cite[Corollary~3.2, Proposition~3.16 and Lemma~5.2]{LiuNeufeld2019TAMS} and~\cite[Corollary~III.2.8]{JacodShiryaev2003}.

\begin{lem}\label{V=Vbar}
Let $\Theta \subseteq \mathbb{R}^d \times \mathbb{S}_+^d \times \mathcal{L}$ be closed, convex and satisfy Assumptions~\ref{B} and~\ref{J}, and let $\Gamma_0,\,\Gamma_1\subseteq\mathcal{P}(\mathbb{R}^d)$ be compact. Moreover, let the cost function $L$ satisfy Assumption~\ref{L}. Then the following statements hold true:
\begin{itemize}
\item[(1)] The set $\overline{\mathfrak{P}}_{\Theta}(\Gamma_0, \Gamma_1):=\{\overline{\mathbb{P}}\in\overline{\mathfrak{P}}_{\Theta}\mid \overline{\mathbb{P}}\circ \overline{X}_{0}^{-1}\in\Gamma_0,\,\overline{\mathbb{P}}\circ \overline{X}_{1}^{-1}\in\Gamma_1\}$ is a compact convex set under the weak convergence topology.
\item[(2)] For any $\mu_0,\mu_1\in\mathcal{P}(\mathbb{R}^d)$ and any $\mathbb{P}\in{\mathfrak{P}}_{\Theta}(\mu_0, \mu_1)$, we can construct $\overline{\mathbb{P}}:=
\mathbb{P}\circ(\Psi^{\mathbb{P}})^{-1}\in\overline{\mathfrak{P}}_{\Theta}(\mu_0,\mu_1)$ such that $J(\mathbb{P})=\overline{J}(\overline{\mathbb{P}})$.
\item[(3)] Conversely, for any $\mu_0,\mu_1\in\mathcal{P}(\mathbb{R}^d)$ and any $\overline{\mathbb{P}}\in\overline{\mathfrak{P}}_{\Theta}(\mu_0, \mu_1)$, we can construct $\mathbb{P}:=
\overline{\mathbb{P}}\circ\overline{X}^{-1}\in{\mathfrak{P}}_{\Theta}(\mu_0, \mu_1)$ such that $J(\mathbb{P})\le \overline{J}(\overline{\mathbb{P}})$.
\end{itemize}
\end{lem}

\subsection{Proof of Theorem~\ref{Vsemicontinuous}}\label{proof_Vsemicontinuous}
One of the key techniques for proving Theorem~\ref{Vsemicontinuous} is the Skorokhod representation theorem. We recall it in the form given in Theorem~3.3 of \cite{LiuNeufeld2019TAMS}.
\begin{thm}[Skorokhod representation theorem]\label{Skorokhod}
Let $\{\overline{\mathbb{P}}^{n}\}_{n\in\mathbb{N}}\subseteq\overline{\mathfrak{P}}_{\Theta}$ be a sequence of probability measures on the enlarged space $\overline{\Omega}$ introduced in Section~2.
Assume that $\overline{\mathbb{P}}^{n}\to \overline{\mathbb{P}}^{0}$ weakly in $\mathcal{P}(\overline{\Omega})$.
Then there exist a probability space $([0,1],\mathcal{B}([0,1]),\lambda)$, where $\lambda$ is the Lebesgue measure, and $\overline{\Omega}$-valued random variables $\{z^{n}\}_{n\in\mathbb{N}_0}$ defined on this space such that
\begin{itemize}
\item[(1)] $\lambda\circ(z^{n})^{-1}=\overline{\mathbb{P}}^{n}$ for every $n\in\mathbb{N}_{0}$.
\item[(2)] $z^{n}\to z^{0}$ pointwise (i.e., in the product topology of $\overline{\Omega}$) $\lambda$-a.s.\ as $n\to\infty$.
\end{itemize}
\end{thm}
In particular, if we write
\begin{align*}
z^{n} 
&= \bigl(z^{n,\overline{X}}, z^{n,\overline{B}}, z^{n,\overline{M}^{c}}, z^{n,\overline{M}^{d}}, z^{n,\overline{J}}, z^{n,\overline{C}}, (z^{n,\overline{V}^{i}})_{i\in\mathbb{N}}\bigr) \\
&= \bigl(\overline{X} \circ z^{n}, \overline{B} \circ z^{n}, \overline{M}^{c} \circ z^{n}, \overline{M}^{d} \circ z^{n}, \overline{J} \circ z^{n}, \overline{C} \circ z^{n}, (\overline{V}^{i} \circ z^{n})_{i\in\mathbb{N}}\bigr),\quad n\in\mathbb{N}_0,
\end{align*}
then each component converges $\lambda$-a.s.\ to the corresponding component of $z^{0}$ in its respective topological space.
Recall the definitions in \eqref{bcv}. For each $n\in\mathbb{N}_0$, we set
$z^{n,\overline{b}}=\overline{b}\circ z^n$, $z^{n,\overline{c}}=\overline{c}\circ z^n$,
$z^{n,\overline{v}^{i}}=\overline{v}^{i}\circ z^n$ ($i\in\mathbb{N}$), and $z^{n,\overline{v}}=\{z^{n,\overline{v}^{i}}\}_{i\in\mathbb{N}}$.
\begin{thm}\label{JPlowsemicontinuous}
Under the assumptions of Theorem~\ref{Vsemicontinuous}, the functional
\[
\overline{J}:\overline{\mathfrak{P}}_{\Theta}\to [0,+\infty],\qquad 
\overline{\mathbb{P}}\mapsto \mathbb{E}^{\overline{\mathbb{P}}}\Bigl[\int_{0}^{1}\overline{L}\bigl(t, \overline{X}, \overline{b}_t,\overline{c}_t,\overline{v}_t\bigr)\,\mathrm{d}t\Bigr],
\]
is convex and lower semicontinuous with respect to the weak convergence on $\overline{\mathfrak{P}}_{\Theta}$.
\end{thm}
\begin{proof}
Let $\{\overline{\mathbb{P}}^n\}_{n\in\mathbb{N}}\subseteq\overline{\mathfrak{P}}_{\Theta}$ and $\overline{\mathbb{P}}^{0}\in\overline{\mathfrak{P}}_{\Theta}$ be such that $\overline{\mathbb{P}}^n\to\overline{\mathbb{P}}^{0}$ weakly.
We must show $\overline{J}(\overline{\mathbb{P}}^{0})\le\liminf_{n\to\infty}\overline{J}(\overline{\mathbb{P}}^n)$; without loss of generality we assume the right‑hand side is finite and that the limit inferior is a true limit, otherwise the statement is trivial. Passing to a subsequence, we may also suppose that $\lim_{n\to+\infty}\overline{J}(\overline{\mathbb{P}}^n)<\infty$.
Apply Theorem~\ref{Skorokhod} to $\{\overline{\mathbb{P}}^n\}_{n\in\mathbb{N}_0}$. By the construction of $\overline{\mathfrak{P}}_{\Theta}$ (cf.~\cite[Corollary~3.2, Lemmas~3.6 and 3.7]{LiuNeufeld2019TAMS}) we obtain the following:
\begin{enumerate}
\item The following convergences hold pointwise with respect to $\mathrm{d}t\times\lambda$:
  \[
  \begin{aligned}
  z_t^{n,\overline{B}}\to z_t^{0,\overline{B}} &= \int_0^t z^{0,\overline{b}}_s\,\mathrm{d}s \quad \lambda\text{-a.s.}, \\
  z_t^{n,\overline{C}}\to z_t^{0,\overline{C}} &= \int_0^t z^{0,\overline{c}}_s\,\mathrm{d}s \quad \lambda\text{-a.s.}, \\
  z_t^{n,\overline{V}^i}\to z_t^{0,\overline{V}^i} &= \int_0^t\int_{\mathbb{R}^d} g_i(z)\,\overline{F}_s(z^0)\,\mathrm{d}z\,\mathrm{d}s \quad \lambda\text{-a.s.}
  \end{aligned}
  \]
\item\label{unint} The following sequences are uniformly integrable on the product space $([0,1]\times[0,1],\mathrm{d}t\times\lambda)$:
  \[
  \begin{aligned}
  \{z^{n,\overline{b}}\}_{n\in\mathbb{N}}&\subseteq L^1([0,1]\times[0,1];\mathbb{R}^d),\\
  \{z^{n,\overline{c}}\}_{n\in\mathbb{N}}&\subseteq L^1([0,1]\times[0,1];\mathbb{R}^{d^2}),\\
  \{z^{n,\overline{v}^{i}}\}_{n\in\mathbb{N}}&\subseteq L^1([0,1]\times[0,1];\mathbb{R}), \quad i=1,2,\dots.
  \end{aligned}
  \]
  The same holds for the sequences $\{z^{n,\overline{B}}\}_{n\in\mathbb{N}}$, $\{z^{n,\overline{C}}\}_{n\in\mathbb{N}}$, and $\{z^{n,\overline{V}^{i}}\}_{n\in\mathbb{N}}$.
\end{enumerate}

By the Dunford–Pettis theorem (see~\cite[Theorem~17]{JoeUniformintegrability1991}), uniform integrability implies relative weak compactness in $L^1$. Applying this to each sequence and using a diagonal argument, we can pass to a further subsequence (still denoted by $n$) such that
\begin{alignat*}{2}
& z^{n,\overline{b}} \to \mathfrak{b} \quad && \text{weakly in } L^1([0,1]\times[0,1];\mathbb{R}^d),\\
& z^{n,\overline{c}} \to \mathfrak{c} \quad && \text{weakly in } L^1([0,1]\times[0,1];\mathbb{R}^{d^2}),\\
& z^{n,\overline{v}^i} \to \mathfrak{v}^i \quad && \text{weakly in } L^1([0,1]\times[0,1];\mathbb{R}),\quad i\in\mathbb{N}.
\end{alignat*}
Set $\mathfrak{v}=\{\mathfrak{v}^{i}\}_{i\in\mathbb{N}}$.
Fix $t\in[0,1]$. For any bounded measurable $f:[0,1]\to\mathbb{R}$, Vitali’s convergence theorem yields
\[
\mathbb{E}^{\lambda}\bigl[f z^{0,\overline{B}}_t\bigr]=\lim_{n\to+\infty}\mathbb{E}^{\lambda}\bigl[f z^{n,\overline{B}}_t\bigr]
=\lim_{n\to+\infty}\mathbb{E}^{\lambda}\Bigl[f \int_0^t z^{n,\overline{b}}_s\,\mathrm{d}s\Bigr]
=\mathbb{E}^{\lambda}\Bigl[f \int_0^t \mathfrak{b}_s\,\mathrm{d}s\Bigr].
\]
Thus $\int_0^t \mathfrak{b}_s\,\mathrm{d}s = z^{0,\overline{B}}_t$ $\lambda$-a.s., which implies $\mathfrak{b}=z^{0,\overline{b}}$ $\lambda$-a.s.\ (Recall from \eqref{bcF=bcv} that $z^{0,\overline{b}} = \overline{b}^{\overline{\mathbb{P}}^0}$.)
The same argument identifies $\mathfrak{c}=z^{0,\overline{c}}$ and $\mathfrak{v}^{i}=z^{0,\overline{v}^{i}}$. In summary,
\begin{equation}\label{identify}
(\mathfrak{b},\mathfrak{c},\mathfrak{v})
=\bigl( \overline{b}, \overline{c}, \overline{v} \bigr)\circ z^{0}
=\Bigl(\overline{b}^{\overline{\mathbb{P}}^0}, \overline{c}^{\overline{\mathbb{P}}^0}, \bigl( \int_{\mathbb{R}^d} g_i(z)\,\overline{F}^{\overline{\mathbb{P}}^0}(\mathrm{d}z) \bigr)_{i \in \mathbb{N}} \Bigr)\circ z^{0},
\qquad \mathrm{d}t\times\lambda\text{-a.s.}
\end{equation}
By Assumption~\ref{L}, the cost function
\[
\overline{L}: (t,\omega,b,c,v)\in[0,1]\times\Omega\times\overline{\varphi}(\Theta) \mapsto \overline{L}(t,\omega,b,c,v)\in[0,+\infty),
\]
is nonnegative, jointly measurable, jointly lower semicontinuous in $(\omega,b,c,v)$ and jointly convex in $(b,c,v)$.
We extend $\overline{L}$ to the whole product space $[0,1]\times\Omega\times\mathbb{R}^d\times\mathbb{R}^{d^2}\times\mathbb{R}^{\mathbb{N}}$ by setting
\[
\overline{L}(t,\omega,b,c,v) = +\infty \quad\text{for } (b,c,v) \notin \overline{\varphi}(\Theta).
\] 
By Remark~\ref{Thetacompact}, under Assumptions~\ref{B} and~\ref{J}, the closed convex set $\Theta$ is compact. Hence $\overline{\varphi}(\Theta)$ is convex and compact. The extended function, still denoted by $\overline{L}$, remains jointly lower semicontinuous in $(\omega,b,c,v)$, jointly convex in $(b,c,v)$, and coincides with the original $\overline{L}$ on $\overline{\varphi}(\Theta)$.
In order to apply Theorem~4.9 of~\cite{Balder1986}, we introduce the following framework. Let $T:=([0,1]\times[0,1],\mathcal{B}([0,1]\times[0,1]),\lambda(\mathrm{d}x)\,\mathrm{d}t)$ and consider the locally convex Polish space
\[
V:=\mathbb{R}^d \times \mathbb{R}^{d^2} \times \mathbb{R}^{\mathbb{N}},
\]
where each finite‑dimensional factor carries its Euclidean topology and $\mathbb{R}^{\mathbb{N}}$ carries the product topology.
Its strong dual (the topological dual equipped with the strong topology, which coincides with the Mackey topology) can be identified algebraically and topologically with
\[
P:=\mathbb{R}^d \times \mathbb{R}^{d^2} \times c_{00},
\]
where $c_{00}$ denotes the space of real sequences with only finitely many non‑zero terms, endowed with the locally convex direct sum topology (the LF topology). This dual space is a locally convex Souslin space (see~\cite[Proposition~A.9]{Treves1967TVP}). The pairing
\[
\langle (b,c,\{v^i\}_{i\in\mathbb{N}}),\,(p,\Gamma,\{f^i\}_{i\in\mathbb{N}}) \rangle = b\cdot p + c\cdot \Gamma + \sum_{i=1}^\infty f^i v^i,
\]
is separating, so $(V,P,\langle \cdot,\cdot\rangle)$ forms a strict duality.
Let $\mathcal{V}$ be the decomposable space of equivalence classes of functions from $T$ to $V$ that are scalarly $\lambda(\mathrm{d}x)\,\mathrm{d}t$-integrable, i.e.\ functions $\vartheta = (\overline{b},\overline{c},\{\overline{v}^i\}_{i\in\mathbb{N}})$ such that for every $\varrho = (p,\Gamma,\{f^i\}_{i\in\mathbb{N}})\in P$ the real function $(x,t)\mapsto \langle\vartheta_t(x),\varrho\rangle$ is $\mathrm{d}t\,\mathrm{d}\lambda(x)$-integrable.
Let $\mathcal{P}$ be the decomposable space of equivalence classes of functions $\varrho = (p,\Gamma,\{f^i\}_{i\in\mathbb{N}})$ from $T$ to $P$ with the property that $p\in L^\infty(T;\mathbb{R}^d)$, $\Gamma\in L^\infty(T;\mathbb{R}^{d^2})$, each $f^i\in L^\infty(T;\mathbb{R})$, and only finitely many of the $f^i$ are not identically zero. The duality between $\mathcal{V}$ and $\mathcal{P}$ is given by
\[
\langle \vartheta,\varrho\rangle_{\mathcal{V},\mathcal{P}} = \mathbb{E}^{\lambda}\Bigl[ \int_0^1 \langle \vartheta_t(x),\varrho_t(x)\rangle\,\mathrm{d}t\Bigr],
\]
which is well defined and strict.
For an arbitrary $\varrho=(p,\Gamma,\{f^i\}_{i\in\mathbb{N}})\in \mathcal{P}$,
we have
\begin{align*}
\langle (z^{n,\overline{b}},z^{n,\overline{c}},
&\{z^{n,\overline{v}^i}\}_{i\in\mathbb{N}}),\,(p,\Gamma,\{f^i\}_{i\in\mathbb{N}})\rangle\\
&=\mathbb{E}^{\lambda}\biggl[\int_0^1\Bigl(z_t^{n,\overline{b}}(x)\cdot p_t(x) + z_t^{n,\overline{c}}(x)\cdot\Gamma_t(x)+\sum_{i\in I}z_t^{n,\overline{v}^i}(x)\!\cdot f_t^i(x)\Bigr)\,\mathrm{d}t\biggr]\\
&=\mathbb{E}^{\lambda}\Bigl[\int_0^1 z_t^{n,\overline{b}}(x)\cdot p_t(x)\,\mathrm{d}t\Bigr]+\mathbb{E}^{\lambda}\Bigl[\int_0^1 z_t^{n,\overline{c}}(x)\cdot\Gamma_t(x)\,\mathrm{d}t\Bigr]+\sum_{i\in I}\mathbb{E}^{\lambda}\Bigl[\int_0^1z_t^{n,\overline{v}^i}(x)\!\cdot f_t^i(x)\,\mathrm{d}t\Bigr],
\end{align*}
for some finite set $I\subset\mathbb{N}$. Recalling the uniform integrability \eqref{unint} and using Vitali’s convergence theorem again, we obtain
\begin{align*}
\lim_{n\to+\infty}
&\langle (z^{n,\overline{b}},z^{n,\overline{c}},\{z^{n,\overline{v}^i}\}_{i\in\mathbb{N}}),\,(p,\Gamma,\{f^i\}_{i\in\mathbb{N}})\rangle\\
&=\mathbb{E}^{\lambda}\Bigl[\int_0^1 z_t^{0,\overline{b}}(x)\cdot p_t(x)\,\mathrm{d}t\Bigr]+\mathbb{E}^{\lambda}\Bigl[\int_0^1 z_t^{0,\overline{c}}(x)\cdot\Gamma_t(x)\,\mathrm{d}t\Bigr]+\sum_{i\in I}\mathbb{E}^{\lambda}\Bigl[\int_0^1 z_t^{0,\overline{v}^i}(x)\!\cdot f_t^i(x)\,\mathrm{d}t\Bigr]\\
&=\langle (z^{0,\overline{b}},z^{0,\overline{c}},\{z^{0,\overline{v}^i}\}_{i\in\mathbb{N}}),\,(p,\Gamma,\{f^i\}_{i\in\mathbb{N}})\rangle.
\end{align*}
Finally, we need to verify the Nagumo tightness of the sequence $(z^{n,\overline{b}},z^{n,\overline{c}},\{z^{n,\overline{v}^i}\}_{i\in\mathbb{N}})_{n\in\mathbb{N}_0}$.
Uniform integrability implies, via the de la Vallée Poussin criterion, the existence of non‑decreasing continuous convex functions $\varphi_b,\varphi_c,\varphi_i : [0,+\infty)\to[0,+\infty)$ with
$$\lim_{r\to\infty}\frac{\min\{\varphi_b(r),\varphi_c(r),\varphi_i(r)\}}{r}=+\infty,$$
and such that
\[
\sup_{n\in\mathbb{N}_0} \int_{[0,1]\times[0,1]} \varphi_b(\|z^{n,\overline{b}}_t(x)\|)\,\lambda(\mathrm{d}x)\,\mathrm{d}t \le 1,\;
\sup_{n\in\mathbb{N}_0} \int_{[0,1]\times[0,1]} \varphi_c(\|z^{n,\overline{c}}_t(x)\|)\,\lambda(\mathrm{d}x)\,\mathrm{d}t \le 1,
\]
\[
\sup_{n\in\mathbb{N}_0} \int_{[0,1]\times[0,1]} \varphi_i(|z^{n,\overline{v}^i}_t(x)|)\,\lambda(\mathrm{d}x)\,\mathrm{d}t \le 2^{-i}
\]
(the latter after a possible rescaling). Define a function $h : V \to [0,+\infty]$ by
$$h(b,c,\{v^i\}_{i\in\mathbb{N}}) = \varphi_b(\|b\|) + \varphi_c(\|c\|) + \sum_{i=1}^{\infty} \varphi_i(|v^i|).$$
The function $h$ is convex and lower semicontinuous on $V$. For any slope $\varrho = (p,\Gamma,\{f^i\})\in P$ and $\gamma\in\mathbb{R}$ the sub‑level set
$\{\nu\in V : h(\nu) - \langle \nu,\varrho\rangle \le \gamma \}$
is closed and each coordinate is bounded; by Tychonoff’s theorem it is compact. Hence $h$ is of Nagumo type on $V$.
Now compute
\[
\begin{aligned}
\sup_{n\in\mathbb{N}_0} \int_{[0,1]\times[0,1]}& h\bigl(z_t^{n,\overline{b}}(x),z_t^{n,\overline{c}}(x),\{z_t^{n,\overline{v}^i}(x)\}_{i\in\mathbb{N}}\bigr)\,\lambda(\mathrm{d}x)\,\mathrm{d}t \\
&\le \sup_{n\in\mathbb{N}_0} \int_{[0,1]\times[0,1]} \varphi_b(\|z^{n,\overline{b}}_t(x)\|)\,\lambda(\mathrm{d}x)\,\mathrm{d}t
 + \sup_{n\in\mathbb{N}_0} \int_{[0,1]\times[0,1]} \varphi_c(\|z^{n,\overline{c}}_t(x)\|)\,\lambda(\mathrm{d}x)\,\mathrm{d}t \\
&\quad + \sum_{i=1}^{\infty} \sup_{n\in\mathbb{N}_0} \int_{[0,1]\times[0,1]} \varphi_i(|z^{n,\overline{v}^i}_t(x)|)\,\lambda(\mathrm{d}x)\,\mathrm{d}t \\
&\le 1 + 1 + \sum_{i=1}^{\infty} 2^{-i} = 3 < +\infty .
\end{aligned}
\]
Therefore $(z^{n,\overline{b}},z^{n,\overline{c}},\{z^{n,\overline{v}^i}\}_{i\in\mathbb{N}})_{n\in\mathbb{N}_0}$ is a Nagumo tight subset of $\mathcal{V}$.
Applying~\cite[Theorem~4.9]{Balder1986}, we conclude that
\begin{align*}
\overline{J}(\overline{\mathbb{P}}^0)
&=\mathbb{E}^{\lambda}\Bigl[\int_{0}^{1}\overline{L}\bigl(t,z^{0,\overline{X}},z_t^{0,\overline{b}},z_t^{0,\overline{c}},z_t^{0,\overline{v}}\bigr)\,\mathrm{d}t\Bigr]\\
&\leq\liminf_{n\to+\infty}\mathbb{E}^{\lambda}\Bigl[\int_{0}^{1}\overline{L}\bigl(t,z^{n,\overline{X}},z_t^{n,\overline{b}},z_t^{n,\overline{c}},z_t^{n,\overline{v}}\bigr)\,\mathrm{d}t\Bigr]\\
&=\liminf_{n\to+\infty}\mathbb{E}^{\overline{\mathbb{P}}^n}\Bigl[\int_{0}^{1}\overline{L}\bigl(t, \overline{X}, \overline{b}_t,\overline{c}_t,\overline{v}_t\bigr)\,\mathrm{d}t\Bigr]
=\liminf_{n\to+\infty}\overline{J}(\overline{\mathbb{P}}^n).
\end{align*}
This completes the proof.
\end{proof}
We have now collected all the ingredients for the proof of Theorem~\ref{Vsemicontinuous}.
\begin{proof}[Proof of Theorem~\ref{Vsemicontinuous}]
The convexity of $V$ can be proved exactly as in~\cite[Lemma~3.15]{TanTouzi2013AOP} (see also the convexity argument in the proof of Theorem~8.3 in~\cite{BVPammer2022Bernoulli}).
We therefore concentrate on the lower semicontinuity.
Let $(\mu_0^n)_{n\in\mathbb{N}},(\mu_1^n)_{n\in\mathbb{N}}\subseteq\mathcal{P}(\mathbb{R}^d)$ satisfy $\mu_0^n\to\mu_0$ and $\mu_1^n\to\mu_1$. If $\liminf_{n}V(\mu_0^n,\mu_1^n)=+\infty$, there is nothing to prove. Hence we assume the limit inferior is finite and, passing to a subsequence, that it is a true limit and that for each $n$ there exists
\[
\mathbb{P}^n\in\mathfrak{P}_{\Theta}(\mu_0^n,\mu_1^n)\quad\text{with}\quad
J(\mathbb{P}^n)\le V(\mu_0^n,\mu_1^n)+\frac1n.
\]
For each $n$, lift $\mathbb{P}^n$ to the enlarged space by $\overline{\mathbb{P}}^n:=\mathbb{P}^n \circ(\Psi^{\mathbb{P}^{n}})^{-1}\in\overline{\mathfrak{P}}_{\Theta}(\mu_0^n,\mu_1^n)$,
then $\overline{J}(\overline{\mathbb{P}}^n)=J(\mathbb{P}^{n})$ by Lemma~\ref{V=Vbar}.
By the compactness result in Lemma~\ref{V=Vbar} (or its proof), the family $\{\overline{\mathbb{P}}^n\}_{n\in\mathbb{N}}$ is tight. Passing to a further subsequence (still indexed by $n$) and noting that $\overline{X}$ has no fixed time of discontinuity (see~\cite[Subsection~3.2]{LiuNeufeld2019TAMS}), we obtain a weak limit point $\overline{\mathbb{P}}\in\overline{\mathfrak{P}}_{\Theta}(\mu_0,\mu_1)$.
Now apply Theorem~\ref{JPlowsemicontinuous} to obtain that
\[
\overline{J}(\overline{\mathbb{P}})\leq\liminf_{n\to\infty}\overline{J}(\overline{\mathbb{P}}^n)
=\liminf_{n\to\infty} J(\mathbb{P}^{n}).
\]
Using Lemma~\ref{V=Vbar} again, the push‑forward measure $\mathbb{P}:=\overline{\mathbb{P}}\circ \overline{X}^{-1}$ belongs to $\mathfrak{P}_{\Theta}(\mu_0,\mu_1)$ with $J(\mathbb{P})\leq\overline{J}(\overline{\mathbb{P}})$.
Consequently,
\[
V(\mu_0,\mu_1)\leq J(\mathbb{P}) \leq\overline{J}(\overline{\mathbb{P}}) \leq \liminf_{n\to\infty} J(\mathbb{P}^{n})
\leq \liminf_{n\to\infty} V(\mu_0^n,\mu_1^n).
\]
Since the sequences $(\mu_0^n,\mu_1^n)$ were arbitrary, $V$ is lower semicontinuous on $\mathcal{P}(\mathbb{R}^d)\times\mathcal{P}(\mathbb{R}^d)$.
\end{proof}

\section{Dynamic programming and viscosity solution formulation}\label{sectmeasele}
In this section we assume that
\begin{equation}\label{definel}
L(t,\omega,\theta)=\ell(t,\omega_t,\theta),
\end{equation}
where $\ell:[0,1]\times\mathbb{R}^d\times\Theta\to\mathbb{R}^{+}$ is a deterministic cost function. Then the function $\lambda_0^{\lambda_1}$ (denoted by $\lambda(0,\cdot)$ for brevity) in \eqref{Rc} reduces to the value function of a standard Markovian stochastic control problem:
\begin{equation}
\lambda(0,x):=\inf_{\mathbb{P}\in\mathfrak{P}_{\Theta}(\delta_x)}\mathbb{E}^{\mathbb{P}}\Bigl[\int_{0}^{1} \ell\bigl(t, X_t, b_t^{\mathbb{P}}, c_t^{\mathbb{P}}, F_t^{\mathbb{P}}\bigr) \, \mathrm{d}t+\lambda_1(X_1) \Bigr].
\end{equation}
Our main objective is to characterise $\lambda$ by means of the corresponding HJB equation.
Let $\Theta$ be Borel measurable. For $t\in[0,1]$, define
\begin{equation}
\begin{aligned}
\mathfrak{P}_{\Theta}^{t,x}
:=\bigl\{\mathbb{P}\in\mathfrak{P}_{sem}^{ac}\mid
&\ (b^{\mathbb{P}},c^{\mathbb{P}},F^{\mathbb{P}})\in\Theta\ \text{on}\ \Omega\times[t,1]\ 
\mathbb{P}\times\mathrm{d}t\text{-a.s. and }\\
&\ \mathbb{P}[X_s=x,\ 0\le s\le t]=1\bigr\}.
\end{aligned}
\end{equation}
The dynamic value function is defined, for any $\lambda_1\in C_b(\mathbb{R}^d)$, by
\begin{equation}
\lambda(t,x):=\inf_{\mathbb{P}\in\mathfrak{P}_{\Theta}^{t,x}}\mathbb{E}^{\mathbb{P}}
\Bigl[\int_{t}^{1} \ell(s,X_s,b_s^{\mathbb{P}},c_s^{\mathbb{P}},F_s^{\mathbb{P}})\,\mathrm{d}s
+\lambda_1(X_1)\Bigr].
\end{equation}
We also introduce the corresponding probability measures on the enlarged space $\overline{\Omega}$. For every $(t,x,b,c,v)\in [0,1]\times\mathbb{R}^d\times\mathbb{R}^d\times\mathbb{R}^{d^2}\times\mathbb{R}^{\mathbb{N}}$, let 
\begin{equation}
\begin{aligned}
\overline{\mathfrak{P}}_{\Theta}^{t,x,b,c,v}
:=\Bigl\{\overline{\mathbb{P}}\in\overline{\mathfrak{P}}_{sem}^{ac}\mid
&\,(\overline{b}_s,\overline{c}_s,\overline{v}_s)\in\overline{\varphi}(\Theta)\ 
\text{on}\ \overline{\Omega}\times[t,1]\ \overline{\mathbb{P}}\times\mathrm{d}t\text{-a.s. and }\\
&\,\overline{\mathbb{P}}\bigl[(\overline{X}_s,\overline{B}_s,\overline{C}_s,\{\overline{V}_s^i\}_{i\in\mathbb{N}})=(x,b,c,v),\, 0\le s\le t\bigr]=1\Bigr\}.
\end{aligned}
\end{equation}
By arguments similar to those in Lemma~\ref{V=Vbar}, if $\Theta \subseteq \mathbb{R}^d \times \mathbb{S}_+^d \times \mathcal{L}$ is closed, convex and satisfies Assumption~\ref{B} and~\ref{J}, and the cost function $\ell$ satisfies Assumption~\ref{L}, then
\begin{equation}\label{liftlambda}
\lambda(t,x)=\inf_{\overline{\mathbb{P}}\in\overline{\mathfrak{P}}_{\Theta}^{t,x,b,c,v}}\mathbb{E}^{\overline{\mathbb{P}}}
\Bigl[\int_{t}^{1} \overline{\ell}(s,\overline{X}_s,\overline{b}_s,\overline{c}_s,\overline{v}_s)\,\mathrm{d}s+\lambda_1(\overline{X}_1)\Bigr]
\end{equation}
holds for all $(b,c,v)\in\mathbb{R}^d\times\mathbb{R}^{d^2}\times\mathbb{R}^{\mathbb{N}}$.
Next, we present the dynamic programming principle. This principle has already been established by Tan and Touzi~\cite{TanTouzi2013AOP} in the continuous case. We give a complete account of the version needed here.
\begin{thm}\label{DPP}
Let $\Theta \subseteq \mathbb{R}^d \times \mathbb{S}_+^d \times \mathcal{L}$ be closed, convex, and satisfy Assumptions~\ref{B} and~\ref{J}. Moreover, let the cost function $\ell$ satisfy Assumption~\ref{L}. Then for every $\overline{\mathbb{F}}$-stopping time $\overline{\tau}$ taking values in $[t,1]$ and all $(b,c,v)\in\mathbb{R}^d\times\mathbb{R}^{d^2}\times\mathbb{R}^{\mathbb{N}}$, we have
\begin{equation}
\lambda(t,x)=\inf_{\overline{\mathbb{P}}\in\overline{\mathfrak{P}}_{\Theta}^{t,x,b,c,v}}\mathbb{E}^{\overline{\mathbb{P}}}
\Bigl[\int_{t}^{\overline{\tau}} \overline{\ell}(s,\overline{X}_s,\overline{b}_s,\overline{c}_s,\overline{v}_s)\,\mathrm{d}s+\lambda(\overline{\tau},\overline{X}_{\overline{\tau}})\Bigr].
\end{equation}
\end{thm}
Before giving the proof, we recall the following stability properties of the class $\overline{\mathfrak{P}}_{\Theta}^{t,x,b,c,v}$ under conditioning and concatenation at stopping times. These facts are essentially contained in~\cite[Theorem~2.1]{NeufeldNutz2017TAMS} and are implicit in~\cite{LiuNeufeld2019TAMS}.
\begin{de}
For any $\overline{\omega},\tilde{\omega}\in\overline{\Omega}$ and any $\overline{\mathbb{F}}$-stopping time $\overline{\tau}$ taking values in $[0,1]$, the concatenation $\overline{\omega}\otimes_{\overline{\tau}(\overline{\omega})}\tilde{\omega}$ is the path defined by
\[
(\overline{\omega}\otimes_{\overline{\tau}(\overline{\omega})}\tilde{\omega})_u := 
\begin{cases}
\overline{\omega}_u, & 0\le u < \overline{\tau}(\overline{\omega}),\\
\overline{\omega}_{\overline{\tau}(\overline{\omega})} + \tilde{\omega}_{u-\overline{\tau}(\overline{\omega})}, & u \geq \overline{\tau}(\overline{\omega}).
\end{cases}
\]
\end{de}
\begin{lem}\label{conditioningandposting}
Let $\Theta \subseteq \mathbb{R}^d \times \mathbb{S}_+^d \times \mathcal{L}$ be measurable. For any $\overline{\mathbb{P}} \in \overline{\mathfrak{P}}_{\Theta}^{t,x,b,c,v}$ and any $\overline{\mathbb{F}}$-stopping time $\overline{\tau}$ taking values in $[t,1]$, the following hold.
\begin{itemize}
\item[(1)]\text{(Conditioning property.)} There exists a family of regular conditional probability measures $(\overline{\mathbb{P}}_{\overline{\omega}})_{\overline{\omega}\in\overline{\Omega}}$ of $\overline{\mathbb{P}}$ given $\overline{\mathcal{F}}_{\overline{\tau}}$ such that for $\overline{\mathbb{P}}$-almost every $\overline{\omega} \in \overline{\Omega}$,
\[
\overline{\mathbb{P}}^{\overline{\tau},\overline{\omega}}(\cdot)
:=\overline{\mathbb{P}}_{\overline{\omega}}\bigl(\overline{\omega}\otimes_{\overline{\tau}}\tilde{\omega}\mid\tilde{\omega}\in\cdot\bigr)
\]
belongs to 
$\overline{\mathfrak{P}}_{\Theta}^{\overline{\tau}(\overline{\omega}), \overline{\omega}_{\overline{\tau}(\overline{\omega})}}$.
\item[(2)]\text{(Concatenation property.)} Let $\{\overline{\mathbb{Q}}_{\overline{\omega}}\}_{\overline{\omega}\in\overline{\Omega}}$ be a family of probability measures such that
\[
\overline{\mathbb{Q}}_{\overline{\omega}} \in \overline{\mathfrak{P}}_{\Theta}^{\overline{\tau}(\overline{\omega}), \overline{\omega}_{\overline{\tau}(\overline{\omega})}}
\quad \text{for } \overline{\mathbb{P}}\text{-a.e. } \overline{\omega},
\]
and the mapping $\overline{\omega} \mapsto \overline{\mathbb{Q}}_{\overline{\omega}}$ is $\overline{\mathcal{F}}_{\overline{\tau}}$-measurable. Then the concatenated measure
\[
(\overline{\mathbb{P}} \otimes_{\overline{\tau}} \overline{\mathbb{Q}})(\cdot) := \int_{\overline{\Omega}} \overline{\mathbb{Q}}_{\overline{\omega}}(\cdot) \, \overline{\mathbb{P}}(\mathrm{d}\overline{\omega})
\]
satisfies $\overline{\mathbb{P}} \otimes_{\overline{\tau}} \overline{\mathbb{Q}} \in \overline{\mathfrak{P}}_{\Theta}^{t,x,b,c,v}$.
\end{itemize}
\end{lem}
Next, we establish a measurable selection result in our framework.
\begin{lem}\label{measele}
Let $\Theta$ be measurable. Moreover, let the cost function $\ell$ satisfy Assumption~\ref{L}. Then, for any probability measure $\mu$ on $\bigl([0,1]\times\mathbb{R}^d\times\mathbb{R}^d\times\mathbb{R}^{d^2}\times\mathbb{R}^{\mathbb{N}},
\mathcal{B}([0,1]\times\mathbb{R}^d\times\mathbb{R}^d\times\mathbb{R}^{d^2}\times\mathbb{R}^{\mathbb{N}})\bigr)$, the following hold:
\begin{itemize}
\item[(1)] The map $(t,x,b,c,v)\mapsto\lambda(t,x)$ is measurable with respect to the $\mu$-completion of $\mathcal{B}([0,1]\times\mathbb{R}^d\times\mathbb{R}^d\times\mathbb{R}^{d^2}\times\mathbb{R}^{\mathbb{N}})$.
\item[(2)] For any $\varepsilon>0$, there exists a family of probability measures
$\{\overline{\mathbb{Q}}^{\varepsilon,t,x,b,c,v}\}$ indexed by $(t,x,b,c,v)\in[0,1]\times\mathbb{R}^d\times\mathbb{R}^d\times\mathbb{R}^{d^2}\times\mathbb{R}^{\mathbb{N}}$ belonging to $\overline{\mathfrak{P}}_{\Theta}^{t,x,b,c,v}$ such that $(t,x,b,c,v)\mapsto\overline{\mathbb{Q}}^{\varepsilon,t,x,b,c,v}$ is measurable with respect to the same completion and 
\begin{equation*}
\mathbb{E}^{\overline{\mathbb{Q}}^{\varepsilon,t,x,b,c,v}}\Bigl[\int_{t}^{1}
\overline{\ell}(s,\overline{X}_s,\overline{b}_s,\overline{c}_s,\overline{v}_s)\,\mathrm{d}s+\lambda_1(\overline{X}_1)\Bigr]<\lambda(t,x)+\varepsilon.
\end{equation*}
\end{itemize}
\end{lem}
\begin{proof}
Note that
\begin{align*}
\overline{\mathfrak{P}}_{\Theta}^{t,x,b,c,v}
=\overline{\mathfrak{P}}_{sem}^{ac}
&\cap\bigl\{\overline{\mathbb{P}}\in\mathcal{P}(\overline{\Omega})\big|\overline{\mathbb{P}}[(\overline{X}_s,\overline{B}_s,\overline{C}_s,\{\overline{V}_s^i\}_{i\in\mathbb{N}})=(x,b,c,v),\, 0\le s\le t]=1\bigr\}\\
&\quad\cap\Bigl\{\overline{\mathbb{P}}\in\overline{\mathfrak{P}}_{sem}^{ac}\mid
\mathbb{E}^{\overline{\mathbb{P}}}\Bigl[\int_t^1\mathbf{1}_{(\overline{b}_s,\overline{c}_s,\overline{v}_s)\notin\overline{\varphi}(\Theta)}\,\mathrm{d}s\Bigr]=0\Bigr\}.
\end{align*}
The Borel measurability of the first and third sets on the right-hand side follows from \cite[Theorem~2.6]{NeufeldNutz2014SPA} and \cite[Corollary~2.7]{NeufeldNutz2014SPA}, respectively, and the second set is closed and therefore Borel. Thus, $\overline{\mathfrak{P}}_{\Theta}^{t,x,b,c,v}$ is Borel as an intersection of Borel sets. Set $E=[0,1]\times\mathbb{R}^d\times\mathbb{R}^d\times\mathbb{R}^{d^2}\times\mathbb{R}^{\mathbb{N}}$ and $F=\mathcal{P}(\overline{\Omega})$. Then the result follows easily from \cite[Theorem~A.1]{TanTouzi2013AOP}.
\end{proof}
We now prove the dynamic programming principle.
\begin{proof}[Proof of Theorem \ref{DPP}]
Let $\overline{\tau}$ be an $\overline{\mathbb{F}}$-stopping time taking values in $[t,1]$. We proceed in two steps.
\begin{itemize}
\item[(1)] For $\overline{\mathbb{P}}\in\overline{\mathfrak{P}}_{\Theta}^{t,x,b,c,v}$, let $(\overline{\mathbb{P}}_{\overline{\omega}})_{\overline{\omega}\in\overline{\Omega}}$ be a family of regular conditional probability distributions of $\overline{\mathbb{P}}$ given $\overline{\mathcal{F}}_{\overline{\tau}}$. By the tower property and Lemma~\ref{conditioningandposting} (1),
\begin{align*}
\lambda(t,x)
&=\inf_{\overline{\mathbb{P}}\in\overline{\mathfrak{P}}_{\Theta}^{t,x,b,c,v}}\mathbb{E}^{\overline{\mathbb{P}}}
\Bigl[\int_{t}^{\overline{\tau}}\overline{\ell}(s,\overline{X}_s,\overline{b}_s,\overline{c}_s,\overline{v}_s)\,\mathrm{d}s
+\int_{\overline{\tau}}^{1} \overline{\ell}(s,\overline{X}_s,\overline{b}_s,\overline{c}_s,\overline{v}_s)\,\mathrm{d}s
+\lambda_1(\overline{X}_1)\Bigr]\\
&=\inf_{\overline{\mathbb{P}}\in\overline{\mathfrak{P}}_{\Theta}^{t,x,b,c,v}}\mathbb{E}^{\overline{\mathbb{P}}} 
\Bigl[\int_{t}^{\overline{\tau}}\overline{\ell}(s,\overline{X}_s,\overline{b}_s,\overline{c}_s,\overline{v}_s)\,\mathrm{d}s
+\mathbb{E}^{\overline{\mathbb{P}}^{\overline{\tau},\overline{\omega}}}\Bigl\{\int_{\overline{\tau}}^1 \overline{\ell}(s,\overline{X}_s,\overline{b}_s,\overline{c}_s,\overline{v}_s)\,\mathrm{d}s
+\lambda_1(\overline{X}_1)\Bigr\}\Bigr]\\
&\geq\inf_{\overline{\mathbb{P}}\in\overline{\mathfrak{P}}_{\Theta}^{t,x,b,c,v}}
\mathbb{E}^{\overline{\mathbb{P}}}\Bigl[\int_{t}^{\overline{\tau}}\overline{\ell}(s,\overline{X}_s,\overline{b}_s,\overline{c}_s,\overline{v}_s)\,\mathrm{d}s+ \lambda (\overline{\tau},\overline{X}_{\overline{\tau}}) \Bigr],
\end{align*}
where the last inequality uses that $\overline{\mathbb{P}}^{\overline{\tau},\overline{\omega}} \in \overline{\mathfrak{P}}_{\Theta}^{\overline{\tau}(\overline{\omega}),\overline{\omega}_{\overline{\tau}(\overline{\omega})}}$ and the definition of $\lambda(t,x)$.
\item[(2)] Fix $\varepsilon>0$ and let $\{\overline{\mathbb{Q}}^{\varepsilon,t,x,b,c,v}\}$ be the family provided by Lemma~\ref{measele}. For each $\overline{\omega}$, set
\[
\overline{\mathbb{Q}}_{\overline{\omega}}^{\varepsilon} := \overline{\mathbb{Q}}^{\,\varepsilon,\,\overline{\tau}(\overline{\omega}),\,\overline{X}_{\overline{\tau}(\overline{\omega})},\,\overline{B}_{\overline{\tau}(\overline{\omega})},
\,\overline{C}_{\overline{\tau}(\overline{\omega})},\,\{\overline{V}_{\overline{\tau}(\overline{\omega})}^{i}\}_{i\in\mathbb{N}}}.
\]
Note that $\overline{\omega}\mapsto\overline{\mathbb{Q}}_{\overline{\omega}}^{\varepsilon}$ is $\overline{\mathcal{F}}_{\overline{\tau}}$-measurable. For any $\overline{\mathbb{P}}\in\overline{\mathfrak{P}}_{\Theta}^{t,x,b,c,v}$, Lemma~\ref{conditioningandposting} (2) yields $\overline{\mathbb{P}}\otimes_{\overline{\tau}}\overline{\mathbb{Q}}^{\varepsilon}\in\overline{\mathfrak{P}}_{\Theta}^{t,x,b,c,v}$. Moreover,
\begin{align*}
&\mathbb{E}^{\overline{\mathbb{P}}\otimes_{\overline{\tau}}\overline{\mathbb{Q}}^{\varepsilon}} 
\Bigl[\int_{t}^{1}\overline{\ell}(s,\overline{X}_s,\overline{b}_s,\overline{c}_s,\overline{v}_s)\,\mathrm{d}s
+\lambda_1(\overline{X}_1) \Bigr]\\
&\qquad = \mathbb{E}^{\overline{\mathbb{P}}}\Bigl[\int_{t}^{\overline{\tau}}\overline{\ell}(s,\overline{X}_s,\overline{b}_s,\overline{c}_s,\overline{v}_s)\,\mathrm{d}s
+ \mathbb{E}^{\overline{\mathbb{Q}}_{\overline{\omega}}^{\varepsilon}}\Bigl\{\int_{\overline{\tau}}^1 \overline{\ell}(s,\overline{X}_s,\overline{b}_s,\overline{c}_s,\overline{v}_s)\,\mathrm{d}s+\lambda_1(\overline{X}_1)\Bigr\}\Bigr] \\
&\qquad \le \mathbb{E}^{\overline{\mathbb{P}}}\Bigl[\int_{t}^{\overline{\tau}}\overline{\ell}(s,\overline{X}_s,\overline{b}_s,\overline{c}_s,\overline{v}_s)\,\mathrm{d}s
+\lambda(\overline{\tau},\overline{X}_{\overline{\tau}}) \Bigr] + \varepsilon.
\end{align*}
\end{itemize}
Since $\overline{\mathbb{P}}\in\overline{\mathfrak{P}}_{\Theta}^{t,x,b,c,v}$ and $\varepsilon>0$ were arbitrary, the reverse inequality follows. This completes the proof.
\end{proof}
The HJB equation is the infinitesimal version of the above dynamic programming principle. Recall the fixed continuous truncation function $h$ and define, for all $(t,x,p,\Gamma,f)\in [0,1]\times\mathbb{R}^d\times\mathbb{R}^d\times\mathbb{S}_{+}^{d}\times C_b^2(\mathbb{R}^d)$,
\begin{equation}\label{H}
H(t,x,p,\Gamma,f(\cdot)):=\inf_{(b,c,F)\in\Theta}\Bigl\{b\cdot p+\frac{1}{2}c\cdot \Gamma+\int_{\mathbb{R}^d}\bigl[f(z)-f(0)-D_{x}f(0)\cdot h(z)\bigr]\,F(\mathrm{d}z)+\ell(t,x,b,c,F)\Bigr\}.
\end{equation}
We next introduce the following continuity assumption.
\begin{ass}\label{ellcontinuous}
The cost function $\ell:[0,1]\times\mathbb{R}^d\times\Theta\to[0,+\infty)$ is nonnegative, jointly continuous in $(t,x,\theta)$, uniformly continuous in $(t,x)$ uniformly in $\theta$, and convex in $\theta$.
\end{ass}
\begin{prop}\label{HJBv}
Let $\Theta \subseteq \mathbb{R}^d \times \mathbb{S}_+^d \times \mathcal{L}$ be closed, convex, and satisfy Assumptions \ref{B} and \ref{J}. Moreover, let $\ell$ satisfy Assumption~\ref{ellcontinuous}. Then the function 
\begin{equation*}
H: [0,1]\times\mathbb{R}^d\times\mathbb{R}^d\times\mathbb{S}_{+}^{d}\times C_b^2(\mathbb{R}^d) \to \mathbb{R}
\end{equation*} 
is continuous.
\end{prop}
\begin{proof}
If $\ell$ is identically zero, the continuity of $H$ is discussed in \cite[Lemma~5.7]{NeufeldNutz2017TAMS}; the uniform continuity of $\ell$ in $(t,x)$ uniformly in $\theta$ extends that argument immediately.
\end{proof}
Next, we present the viscosity solution theorem. When the cost function is identically zero, this result already appeared in~\cite[Corollary~4.2]{Kuhn2019ALEA} (see also~\cite[Theorem~2.5]{NeufeldNutz2017TAMS} and~\cite{HuPeng2009}).

\begin{thm}\label{viscosity}
Let $\Theta \subseteq \mathbb{R}^d \times \mathbb{S}_+^d \times \mathcal{L}$ be closed, convex, and satisfy Assumptions \ref{B} and \ref{J}. Moreover, let $\ell$ satisfy Assumption~\ref{ellcontinuous}. Then $\lambda$ is the unique viscosity solution of the HJB equation
\begin{equation}\label{HJBE}
\begin{cases}
-\partial_{t}\lambda(t,x)-H\bigl(t,x,D_x\lambda(t,x),D_{xx}^2\lambda(t,x),\lambda(t,x+\cdot)\bigr)=0,\\
\lambda(1,x)=\lambda_1(x).
\end{cases}
\end{equation}
\end{thm}
\begin{proof}
The proof that $\lambda$ is a viscosity solution is similar to that of Theorem~4.2 in~\cite{TanTouzi2013AOP}, adapted to our setting.

\noindent\textit{Subsolution property:} Let $(t_0, x_0)\in[0,1)\times\mathbb{R}^d$ and let $\phi\in C_c^{\infty}([0,1)\times\mathbb{R}^d)$ be such that
\[
0 = (\lambda-\phi) (t_0, x_0) > (\lambda- \phi) (t,x),\qquad \forall(t,x) \neq(t_0, x_0).
\]
By adding $\varepsilon (|t-t_0|^2 + |x-x_0|^4)$ to $\phi(t,x)$, we may assume without loss of generality that
\begin{equation}\label{phimaxepsilon}
\phi(t,x) \ge\lambda(t,x) + \varepsilon \bigl(|t-t_0|^2 + |x-x_0|^4 \bigr).
\end{equation}
Suppose, for contradiction, that
\begin{equation*}
- \partial_t \phi(t_0, x_0) - H\bigl(t_0,x_0,D_x \phi(t_0,x_0), D^2_{xx} \phi(t_0,x_0), \phi(t_0,x_0+\cdot) \bigr) > 0.
\end{equation*}
By the definition of $H$, there exist $k > 0$ and $(b,c,F)\in \Theta$ such that
\begin{align*}
-\partial_t \phi(t, x)-b \cdot D_x \phi(t, x)-\frac{1}{2} c \cdot D_{xx}^2 \phi(t, x)- 
\int_{\mathbb{R}^d} \bigl[\phi(t,x+z)-\phi(t,x)-D_x\phi(t,x)\cdot h(z) \bigr] F(\mathrm{d}z) \\
- \ell(t, x, b, c, F)>0\qquad\forall (t, x) \in B_k(t_0, x_0),
\end{align*}
where $B_k(t_0, x_0) := \{(t, x) \in [0, 1) \times \mathbb{R}^d : |(t, x) - (t_0, x_0)| \leq k\}$.
Define $\overline{\tau}:=\inf\{t\in [t_0,1]:(t,\overline{X}_t)\notin B_k(t_0, x_0)\}$. Applying It\^o's formula to $\phi(\overline{\tau}, \overline{X}_{\overline{\tau}})$, we obtain that
\begin{align*}
\phi(\overline{\tau}, \overline{X}_{\overline{\tau}}) - \phi(t_0, x_0)
&= \int_{t_0}^{\overline{\tau}} \partial_t \phi(s,\overline{X}_s)\,\mathrm{d}s + \int_{t_0}^{\overline{\tau}} D_x \phi(s,\overline{X}_{s-}) \, \mathrm{d}\overline{M}_s^{c} \\
&\quad + \int_{t_0}^{\overline{\tau}} \overline{b}_s^{\overline{\mathbb{P}}}\cdot D_x \phi(s,\overline{X}_{s-}) \,\mathrm{d}s + \int_{t_0}^{\overline{\tau}} \frac{1}{2}\overline{c}_s^{\overline{\mathbb{P}}}\cdot D_{xx}^2 \phi(s, \overline{X}_{s-})\,\mathrm{d}s \\
&\quad + \int_{t_0}^{\overline{\tau}} \int_{\mathbb{R}^d}\Bigl[\phi(s,\overline{X}_{s-}+z)-\phi(s,\overline{X}_{s-})\Bigr]\,(\mu^{\overline{X}}-\nu^{\overline{\mathbb{P}}})(\mathrm{d}z,\mathrm{d}s) \\
&\quad + \int_{t_0}^{\overline{\tau}} \int_{\mathbb{R}^d}\Bigl[\phi(s,\overline{X}_{s-}+z)-\phi(s,\overline{X}_{s-})-D_x\phi(s,\overline{X}_{s-})\cdot h(z)\Bigr]\,\overline{F}_s^{\overline{\mathbb{P}}}(\mathrm{d}z)\,\mathrm{d}s,
\end{align*}
where $\overline{M}^c$ is the continuous local martingale part of $\overline{X}$. Note that $\overline{X}$ has no fixed times of discontinuity under $\mathbb{P}$; see~\cite[Lemma~2.3]{LiuNeufeld2019TAMS}. Taking expectations and using the definition of $H$, we get
\begin{align*}
\lambda(t_0,x_0)=\phi(t_0,x_0) 
&\geq\inf_{\overline{\mathbb{P}}\in\overline{\mathfrak{P}}_{\Theta}^{t_0,x_0,0,0,0}}
\mathbb{E}^{\overline{\mathbb{P}}}\Bigl[\int_{t_0}^{\overline{\tau}}\overline{\ell}(s,\overline{X}_s,\overline{b}_s,\overline{c}_s,\overline{v}_s)\,\mathrm{d}s
+\phi(\overline{\tau},\overline{X}_{\overline{\tau}}) \Bigr] \\
&\geq\inf_{\overline{\mathbb{P}}\in\overline{\mathfrak{P}}_{\Theta}^{t_0,x_0,0,0,0}} 
\mathbb{E}^{\overline{\mathbb{P}}}\Bigl[\int_{t_0}^{\overline{\tau}}\overline{\ell}(s,\overline{X}_s,\overline{b}_s,\overline{c}_s,\overline{v}_s)\,\mathrm{d}s+
\lambda(\overline{\tau},\overline{X}_{\overline{\tau}})\Bigr] + \eta,
\end{align*}
where $\eta>0$ is a constant coming from \eqref{phimaxepsilon} and the definition of $\overline{\tau}$. This contradicts Theorem~\ref{DPP}.

\noindent\textit{Supersolution property:} The proof is symmetric to the subsolution case, noting that $-\partial_t \phi - H < 0$ forces the reverse inequality for every $(b,c,F)\in\Theta$ locally, and the contradiction with Theorem~\ref{DPP} follows by It\^o's formula.

\noindent\textit{Uniqueness:} Let $USC_b(\mathbb{R}^d)$ denote the set of bounded upper semicontinuous functions on $\mathbb{R}^d$, $LSC_b(\mathbb{R}^d)$ the bounded lower semicontinuous functions, and $SC_b(\mathbb{R}^d) := USC_b(\mathbb{R}^d) \cup LSC_b(\mathbb{R}^d)$.
For $\kappa>0$ and $(t,x,p,\Gamma,f,g)\in [0,1]\times\mathbb{R}^d\times\mathbb{R}^d\times\mathbb{S}_{+}^{d}\times SC_b(\mathbb{R}^d)\times C^2(\mathbb{R}^d)$, define
\begin{equation*}
\begin{aligned}
H^{\kappa}(t,x,p,\Gamma,f(\cdot),g(\cdot)):=\inf_{(b,c,F)\in\Theta}
\Bigl\{
&b\cdot p+\frac{1}{2}c\cdot\Gamma+\int_{|z|>\kappa}\bigl[f(z)-f(0)-D_{x}g(0)\cdot h(z)\bigr]\,F(\mathrm{d}z)\\
&\quad +\int_{|z|\leq\kappa}\bigl[g(z)-g(0)-D_{x}g(0)\cdot h(z)\bigr]\,F(\mathrm{d}z)+\ell(t,x,b,c,F)\Bigr\}.
\end{aligned}
\end{equation*}
One can verify that $H^{\kappa}$ satisfies conditions (C1)–(C9) in~\cite[Lemma~5.6]{NeufeldNutz2017TAMS} (see also~\cite{HuPeng2009}), and therefore the uniqueness of the viscosity solution follows from the comparison principle.
\end{proof}
\begin{rem}
When $\ell\equiv0$, the problem reduces to a sublinear expectation and the HJB equation~\eqref{HJBE} becomes a nonlinear Kolmogorov equation.
In this case, the time-reversed function $u^{\lambda_1}(t,x):=\lambda(1-t,x)$ induces a sublinear Markovian semigroup $(T_t)_{t\ge0}$ via $T_t f(x)=u^f(t,x)$; it satisfies the semigroup property $T_{s+t}=T_s T_t$ and, under mild conditions, the Feller property (see~\cite{DKN2020SPA}).
This control-theoretic formulation provides a unified foundation for $G$-Brownian motion~\cite{Peng2019}, nonlinear L\'{e}vy processes~\cite{NeufeldNutz2017TAMS,CriensNiemann2024SPA}, and path-dependent nonlinear semimartingales~\cite{CriensNiemann2023EJP}, and links optimal transport and model-free finance to the theory of nonlinear expectations.
\end{rem}
%\begin{rem}
%When the cost function $\ell$ is identically zero, the problem reduces to a sublinear expectation and the associated HJB equation becomes a standard nonlinear Kolmogorov equation. This control-theoretic formulation provides a unified foundation for $G$-Brownian motion~\cite{Peng2007SAA, Peng2008SPA, Peng2019}, nonlinear L\'{e}vy processes~\cite{NeufeldNutz2017TAMS, CriensNiemann2024SPA}, and path-dependent nonlinear semimartingales~\cite{CriensNiemann2023EJP}, while linking optimal transport and model-free finance to the theory of nonlinear expectations.
%\end{rem}
%
%\begin{rem}
%When $\ell\equiv0$, consider the time-reversed function $u^{\lambda_1}(t,x):=\lambda(1-t,x)$. For $t\in[0,1]$ and $x\in\mathbb{R}^d$, define the operator $T_t f(x):=u^{f}(t,x)$. Then $(T_t)$ satisfies the semigroup property $T_{s+t}=T_{s}T_{t}$ and is a sublinear Markovian semigroup; under mild conditions, it also possesses the Feller property. We refer to~\cite{DKN2020SPA} and the references therein for further details.
%\end{rem}

Using the uniqueness of viscosity solutions, we can characterise the semimartingale optimal transport problem completely via the HJB equation.
\begin{thm}
Let $\Theta \subseteq \mathbb{R}^d \times \mathbb{S}_+^d \times \mathcal{L}$ be closed, convex, and satisfy Assumptions \ref{B} and \ref{J}. Moreover, let $\ell$ satisfy Assumption~\ref{ellcontinuous}. Then
\begin{equation*}
V(\mu_0,\mu_1)=\sup_{\lambda_1\in C_b(\mathbb{R}^d)}\Bigl\{\int_{\mathbb{R}^d}\lambda(0,x)\,\mu_0(\mathrm{d}x)-\int_{\mathbb{R}^d}\lambda_1(x)\,\mu_1(\mathrm{d}x)\Big| \lambda \text{ solves } \eqref{HJBE} \text{ with } \lambda(1,\cdot)=\lambda_1\Bigr\}.
\end{equation*}
\end{thm}
For continuous semimartingales, we have the following result based on~\cite[Theorem~4.2]{TanTouzi2013AOP} (note that viscosity solutions may not be unique in this setting).
\begin{prop}
Let $\Theta\subseteq\mathbb{R}^d\times\mathbb{S}_+^d\times\{0\}$ be closed and convex, and let the cost function $\ell$ satisfy Assumptions \ref{ellcontinuous} and \ref{uL}. Then $\lambda$ is a viscosity solution of the HJB equation \eqref{HJBE}.
%, and the equality \eqref{V=HJB} holds.
\end{prop}

\section{Markovian projection and PDE formulation}\label{5}
In this section we establish two further equivalent formulations of the semimartingale optimal transport problem. The first one, developed in Subsection~\ref{Markovian}, employs the Markovian projection technique to replace a general semimartingale by a Markovian one without increasing the cost, thereby reducing~\eqref{SOT} to a martingale problem. The second one, presented in Subsection~\ref{PDESOT}, translates this martingale problem into the language of non-local FPK equations, yielding a PDE formulation that generalises the classical Benamou--Brenier formula.

\subsection{The Markovian semimartingale transport problem}\label{Markovian}
A Markovian projection (or mimicking process) transforms a path-dependent semimartingale into another semimartingale whose differential characteristics depend only on the current state $X_t$, while preserving the one--dimensional marginal laws.  The resulting process solves a Markovian SDE, although the Markov property may fail in general. We follow the convention of~\cite{LarssonLong2024ECP} and rely on~\cite[Theorem~3.2]{LarssonLong2024ECP} for the existence of such a projection under our assumptions.
\begin{lem}\label{Markovianprojection}
Let $\Theta\subseteq\mathbb{R}^d\times\mathbb{S}_+^d\times\mathcal{L}$ be closed, convex and satisfy Assumption~\ref{B}. 
Then for every $\mathbb{P}\in\mathfrak{P}_{\Theta}(\mu_0)$ there exists another probability measure $\widehat{\mathbb{P}}\in\mathfrak{P}_{\Theta}(\mu_0)$ with differential characteristics $(\hat b,\hat c,\widehat F)$ such that
\[
\widehat{\mathbb{P}}\circ X_t^{-1}=\mathbb{P}\circ X_t^{-1}\qquad\forall t\in[0,1],
\]
and under $\widehat{\mathbb{P}}$ the canonical process $X$ solves the martingale problem for the non‑local operator $\mathscr{L}=(\mathscr{L}_s)_{0\leq s\leq 1}$ with initial law $\mu_0$. More precisely,
\begin{enumerate}
\item $\widehat{\mathbb{P}}\circ X_0^{-1}=\mu_0$;
\item for every $f\in C_c^2(\mathbb{R}^d)$, the process
\[
M^f_t:=f(X_t)-f(X_0)-\int_0^t \mathscr{L}_s f(X_s)\,\mathrm{d}s
\]
is well defined and an $\mathbb{F}$‑martingale under $\widehat{\mathbb{P}}$, 
\end{enumerate}
where the operator $\mathscr{L}_s$ acts on $f\in C_c^2(\mathbb{R}^d)$ as
\begin{equation}\label{mathscrL}
\begin{aligned}
\mathscr{L}_s f(x)&:=\hat b(s,x)\cdot D_x f(x)+\frac12\hat c(s,x)\cdot D_{xx}^2 f(x)\\
&\quad +\int_{\mathbb{R}^d}\bigl(f(x+z)-f(x)-D_x f(x)\cdot h(z)\bigr)\,\widehat F(s,x,\mathrm{d}z),\quad \forall s\in[0,1],
\end{aligned}
\end{equation}
with $h$ being the fixed continuous truncation function. The coefficients $(\hat b,\hat c,\widehat F)$ are defined as jointly measurable functions on $[0,1]\times\mathbb{R}^d$ satisfying
\begin{equation}\label{hatbcF}
\begin{aligned}
\hat b(t,X_{t-}) &= \mathbb{E}^{\mathbb{P}}\bigl[b_t^{\mathbb{P}}\,\big|\, X_{t-}\bigr],\\
\hat c(t,X_{t-}) &= \mathbb{E}^{\mathbb{P}}\bigl[c_t^{\mathbb{P}}\,\big|\, X_{t-}\bigr],\\
\int_A (1\land|z|^2)\,\widehat F(t,X_{t-},\mathrm{d}z) &= \mathbb{E}^{\mathbb{P}}\!\left[\int_A (1\land|z|^2)\,F_t^{\mathbb{P}}(\mathrm{d}z)\,\bigg|\, X_{t-}\right],\quad \forall A\in\mathcal{B}(\mathbb{R}^d).
\end{aligned}
\end{equation}
\end{lem}
\begin{rem}\label{g_ix}
Moreover, with a bit more effort (cf.~\cite[Remark~2.6 or (3.6)]{LarssonLong2024ECP}), we can show that
\begin{equation*}
\int_{\mathbb{R}^d} g(X_{t-},z)\,\widehat F(t,X_{t-},\mathrm{d}z) = \mathbb{E}^{\mathbb{P}}\!\left[\int_{\mathbb{R}^d} g(X_{t-},z)\,F_t^{\mathbb{P}}(\mathrm{d}z)\,\bigg|\, X_{t-}\right]
\end{equation*}
holds for every measurable function $g:\mathbb{R}^d\times\mathbb{R}^d\to\mathbb{R}$ satisfying $|g(x,z)|\le C(1\land|z|^2)$ for all $x,z\in\mathbb{R}^d$, with some constant $C>0$.
\end{rem}
\begin{rem}
The truncation function used in~\cite[Theorem~3.2]{LarssonLong2024ECP} is $h(z)=z\mathbf{1}_{\{|z|\le 1\}}$, which is discontinuous.  However, the proof in~\cite{LarssonLong2024ECP} only relies on the boundedness of $h$ and the fact that $h(z)=z$ near the origin.  One may therefore replace $h$ by the smooth, compactly supported truncation function $\pi$ from~\cite[Section~3.3 and (3.15)]{RocknerXieZhang2020PTRF}, which satisfies $\pi(z)=z$ for $|z|\le\delta$ and is continuous. This substitution does not alter any step of the proof, so~\cite[Theorem~3.2]{LarssonLong2024ECP} remains valid.
\end{rem}

For $\mu_0\in\mathcal{P}(\mathbb{R}^d)$, set
\begin{align*}
\widehat{\mathfrak{P}}_{\Theta}(\mu_0)
:=\bigl\{\widehat{\mathbb{P}}\in\mathfrak{P}_{\Theta}(\mu_0)\mid
&\widehat{\mathbb{P}} \text{ is such that } X \text{ solves the martingale problem for }\mathscr{L}\\
&\text{with the Markovian differential characteristics }(\hat{b},\hat{c},\widehat{F}) \text{ defined in } \eqref{hatbcF}\bigr\}.
\end{align*}
For another probability measure $\mu_1\in\mathcal{P}(\mathbb{R}^d)$, define
\[
\widehat{\mathfrak{P}}_{\Theta}(\mu_0,\mu_1):=\bigl\{\widehat{\mathbb{P}}\in\widehat{\mathfrak{P}}_{\Theta}(\mu_0)\;\big|\;
\widehat{\mathbb{P}}\circ X_1^{-1}=\mu_1\bigr\}.
\]
\begin{rem}
If $\Theta$ is closed, convex and satisfies Assumption~\ref{B}, then by definition we have
\begin{align*}
\widehat{\mathfrak{P}}_{\Theta}(\mu_0,\mu_1)
=\bigl\{\widehat{\mathbb{P}}\in\mathfrak{P}_{\Theta}(\mu_0,\mu_1)\mid
&\widehat{\mathbb{P}} \text{ is such that } X \text{ solves the martingale problem for }\mathscr{L} \text{ with the} \\
&\text{Markovian differential characteristics }(\hat{b},\hat{c},\widehat{F}) \text{ defined in } \eqref{hatbcF}\bigr\}.
\end{align*}
Lemma~\ref{Markovianprojection} guarantees that $\widehat{\mathfrak{P}}_{\Theta}(\mu_0,\mu_1)$ is non‑empty whenever $\mathfrak{P}_{\Theta}(\mu_0,\mu_1)$ is non‑empty. 
\end{rem}
\begin{cor}\label{Markovnonempty}
Let $\Theta\subseteq\mathbb{R}^d\times\mathbb{S}_+^d\times\mathcal{L}$ be closed, convex and satisfy Assumption~\ref{B}. If $\mathfrak{P}_{\Theta}(\mu_0,\mu_1)\neq\emptyset$, then $\widehat{\mathfrak{P}}_{\Theta}(\mu_0,\mu_1)\neq\emptyset$. Moreover, for any $\mathbb{P}\in\mathfrak{P}_{\Theta}(\mu_0,\mu_1)$ there exists a measure $\widehat{\mathbb{P}}\in\widehat{\mathfrak{P}}_{\Theta}(\mu_0,\mu_1)\subseteq\mathfrak{P}_{\Theta}(\mu_0,\mu_1)$ such that $X$ has the same one--dimensional marginal laws under $\mathbb{P}$ and $\widehat{\mathbb{P}}$.
\end{cor}
We can then define the Markovian semimartingale transport problem with the cost function $\ell$ defined in \eqref{definel}:
\begin{equation}\label{Markovsemipro}
\widehat{V}(\mu_0,\mu_1):=\inf_{\widehat{\mathbb{P}}\in\widehat{\mathfrak{P}}_{\Theta}(\mu_0,\mu_1)} J(\widehat{\mathbb{P}}),
\end{equation}
with the convention $\inf\emptyset=+\infty$.
We next prove the equivalence between the Markovian semimartingale optimal transport problem and the general semimartingale optimal transport problem, using the fact that a Markovian projection does not increase the cost. The following theorem has been proved in the continuous framework in~\cite{GuoLoeperWang2022MF}.
\begin{thm}\label{V=Vhat}
Let $\Theta\subseteq\mathbb{R}^d\times\mathbb{S}_+^d\times\mathcal{L}$ be closed, convex and satisfy Assumptions~\ref{B} and~\ref{J}. Moreover, let the cost function $\ell$ (defined in \eqref{definel}) satisfy Assumption~\ref{L}. Then we have
\begin{equation*}
V(\mu_0,\mu_1)=\widehat{V}(\mu_0,\mu_1),
\end{equation*}
and the infimum in \eqref{SOT} is attained by some $\widehat{\mathbb{P}}\in\widehat{\mathfrak{P}}_{\Theta}(\mu_0,\mu_1)$ whenever $V(\mu_0,\mu_1)<+\infty$.
\end{thm}
\begin{proof}
If $\mathfrak{P}_{\Theta}(\mu_0,\mu_1)=\emptyset$, then $\widehat{\mathfrak{P}}_{\Theta}(\mu_0,\mu_1)=\emptyset$ because it is a subset of the former. In this case both sides equal $+\infty$ by convention, and the equality holds trivially.
Assume now that $\mathfrak{P}_{\Theta}(\mu_0,\mu_1)$ is non‑empty. Since $\widehat{\mathfrak{P}}_{\Theta}(\mu_0,\mu_1)\subseteq\mathfrak{P}_{\Theta}(\mu_0,\mu_1)$, we clearly have $V(\mu_0,\mu_1)\le\widehat{V}(\mu_0,\mu_1)$. For the reverse inequality, fix an arbitrary $\mathbb{P}\in\mathfrak{P}_{\Theta}(\mu_0,\mu_1)$. By Corollary~\ref{Markovnonempty} there exists a measure $\widehat{\mathbb{P}}\in\widehat{\mathfrak{P}}_{\Theta}(\mu_0,\mu_1)$ such that the one--dimensional marginal laws of $X$ under $\mathbb{P}$ and $\widehat{\mathbb{P}}$ coincide. Moreover, under $\widehat{\mathbb{P}}$ the process $X$ solves the martingale problem for $\mathscr{L}$ with coefficients $(\hat b,\hat c,\widehat F)$.
We now estimate $J(\mathbb{P})$ from below. Recall the cost function $\overline{\ell}$ defined in \eqref{deoverl} and the mapping $\varphi$ defined in \eqref{eq:map}. By Lemma~\ref{Markovianprojection} and Jensen's inequality, we obtain
\begin{align*}
J(\mathbb{P})
&=\mathbb{E}^{\mathbb{P}}\!\left[\int_0^1 \ell\bigl(t,X_t,b_t^{\mathbb{P}},c_t^{\mathbb{P}},F_t^{\mathbb{P}}\bigr)\,\mathrm{d}t\right] \\
&=\mathbb{E}^{\mathbb{P}}\!\left[\int_0^1 \overline{\ell}\Bigl(t,X_t,b_t^{\mathbb{P}},c_t^{\mathbb{P}},
\Bigl\{\int_{\mathbb{R}^d}g_{i}(z)\,F_t^{\mathbb{P}}(\mathrm{d}z)\Bigr\}_{i\in\mathbb{N}}\Bigr)\,\mathrm{d}t\right]\\
&\ge\mathbb{E}^{\mathbb{P}}\!\left[\int_0^1 \overline{\ell}\Bigl(t,X_t,\,
\mathbb{E}^{\mathbb{P}}\!\Bigl[b_t^{\mathbb{P}},c_t^{\mathbb{P}},\Bigl\{\int_{\mathbb{R}^d}g_{i}(z)\,F_t^{\mathbb{P}}(\mathrm{d}z)\Bigr\}_{i\in\mathbb{N}}\,\Big|\,
X_{t-}\Bigr]\Bigr)\,\mathrm{d}t\right]\\
&=\mathbb{E}^{\mathbb{P}}\!\left[\int_0^1 \overline{\ell}\Bigl(t,X_t,\hat b(t,X_{t-}),\hat c(t,X_{t-}),
\mathbb{E}^{\mathbb{P}}\!\Bigl[\Bigl\{\int_{\mathbb{R}^d}g_{i}(z)\,F_t^{\mathbb{P}}(\mathrm{d}z)\Bigr\}_{i\in\mathbb{N}}\,\Big|\,
X_{t-}\Bigr]\Bigr)\,\mathrm{d}t\right].
\end{align*}
Since $\mathcal{C}^+(\mathbb{R}^d):=\{g_i \mid i \in \mathbb{N}\}\subseteq C_b(\mathbb{R}^d)$ is a law-determining class for L\'evy measures, we can apply Remark~\ref{g_ix} to obtain
\begin{align*}
\mathbb{E}^{\mathbb{P}}
&\left[\int_0^1 \overline{\ell}\Bigl(t,X_t,\hat b(t,X_{t-}),\hat c(t,X_{t-}),
\mathbb{E}^{\mathbb{P}}\!\Bigl[\Bigl\{\int_{\mathbb{R}^d}g_{i}(z)\,F_t^{\mathbb{P}}(\mathrm{d}z)\Bigr\}_{i\in\mathbb{N}}\,\Big|\,
X_{t-}\Bigr]\Bigr)\,\mathrm{d}t\right]\\
&=\mathbb{E}^{\mathbb{P}}\!\left[\int_0^1 \overline{\ell}\Bigl(t,X_t,\hat b(t,X_{t-}),\hat c(t,X_{t-}),
\Bigl\{\int_{\mathbb{R}^d}g_{i}(z)\,\widehat F(t,X_{t-},\mathrm{d}z)\Bigr\}_{i\in\mathbb{N}}\Bigr)\,\mathrm{d}t\right]\\
&=\mathbb{E}^{\mathbb{P}}\!\left[\int_0^1 \ell\bigl(t,X_t,\hat b(t,X_{t-}),\hat c(t,X_{t-}),\widehat F(t,X_{t-})\bigr)\,\mathrm{d}t\right].
\end{align*}
Note that $X$ has no fixed time of discontinuity under $\mathbb{P}$ (see~\cite[Lemma~2.9]{LarssonLong2024ECP}). It follows that 
\begin{align*}
\mathbb{E}^{\mathbb{P}}
&\left[\int_0^1 \ell\bigl(t,X_t,\hat b(t,X_{t-}),\hat c(t,X_{t-}),\widehat F(t,X_{t-})\bigr)\,\mathrm{d}t\right]\\
&\quad=\mathbb{E}^{\mathbb{P}}\!\left[\int_0^1 \ell\bigl(t,X_t,\hat b(t,X_{t}),\hat c(t,X_{t}),\widehat F(t,X_{t})\bigr)\,\mathrm{d}t\right]\\
&\quad=\mathbb{E}^{\widehat{\mathbb{P}}}\!\left[\int_0^1 \ell\bigl(t,X_t,\hat b(t,X_{t}),\hat c(t,X_{t}),\widehat F(t,X_{t})\bigr)\,\mathrm{d}t\right].
\end{align*}
The last equality is justified by the fact that $\mathbb{P}$ and $\widehat{\mathbb{P}}$ admit the same one--dimensional marginal laws and the integrand depends on the path only through the state $X_t$. Hence $J(\mathbb{P})\geq J(\widehat{\mathbb{P}})$. Taking the infimum yields the desired result.
\end{proof}

For continuous semimartingales, we have the following result based on~\cite[Proposition~3.5]{GuoLoeperWang2022MF} and Proposition~\ref{F=0SOT}.
\begin{prop}\label{F=0V=Vhat}
Let $\Theta\subseteq\mathbb{R}^d\times\mathbb{S}_+^d\times\{0\}$ be closed and convex, and let the cost function $\ell$ satisfy Assumptions~\ref{L} and~\ref{uL}. Then the conclusions of Theorem~\ref{V=Vhat} also hold.
\end{prop}

\subsection{PDE formulation}\label{PDESOT}
Moreover, by the connections established in Lemma~\ref{Markovianprojection}, the semimartingale optimal transport problem $V(\mu_0,\mu_1)$ defined in \eqref{SOT} can be studied using PDE methods. Following the Benamou--Brenier formulation of classical optimal transport, let $\rho_t$ denote the marginal law of $X_t$. Then, by Dynkin’s formula, we obtain the following non-local FPK equation:
\begin{equation}\label{FPKE}
\partial_{t}\rho_t=\mathscr{L}_t^{*}\rho_t.
\end{equation}
We next define weak solutions of this equation, following the definition in~\cite{RocknerXieZhang2020PTRF}.
\begin{de}[Weak solution]
Let $\Theta\subseteq\mathbb{R}^d\times\mathbb{S}_+^d\times\mathcal{L}$ be closed, convex and satisfy Assumptions~\ref{B} and~\ref{J}, and let $\rho:[0,1]\to\mathcal{P}(\mathbb{R}^d)$ be a continuous curve. We call $\rho=(\rho_t)_{0\leq t\leq 1}$ a weak solution of the non-local FPK equation \eqref{FPKE} induced by a measurable function $\theta=(\theta(t))_{0\leq t\leq 1}$, where $\theta(t,x)=(b(t,x),c(t,x),F(t,x))\in\Theta$ for all $x\in\mathbb{R}^d$, if for all $f\in C_c^2(\mathbb{R}^d)$ and $t\in[0,1]$,
\begin{equation}
\int_{\mathbb{R}^d}f(x)\,\rho_t(\mathrm{d}x)=\int_{\mathbb{R}^d}f(x)\,\rho_0(\mathrm{d}x)+\int_0^t\int_{\mathbb{R}^d}\mathscr{L}_s^{\theta}f(x)\,\rho_s(\mathrm{d}x)\,\mathrm{d}s,
\end{equation}
where 
\begin{equation*}
\begin{aligned}
\mathscr{L}_s^{\theta} f(x)&:= b(s,x)\cdot D_x f(x)+\frac{1}{2} c(s,x)\cdot D_{xx}^2 f(x)\\
&\quad +\int_{\mathbb{R}^d}\bigl(f(x+z)-f(x)-D_x f(x)\cdot h(z)\bigr)\, F(s,x,\mathrm{d}z),\quad \forall s\in[0,1].
\end{aligned}
\end{equation*}
\end{de}
We next introduce the following PDE optimisation problem:
\begin{equation}\label{P}
PDE(\mu_0,\mu_1):=\inf_{\theta,\rho}\Bigl\{\int_0^1\int_{\mathbb{R}^d} \ell(t,x,\theta(t,x))\,\rho_t(\mathrm{d}x)\,\mathrm{d}t\Bigr\},
\end{equation}
with the convention $\inf\emptyset=+\infty$, where the infimum is taken over all pairs $(\theta,\rho)$ satisfying the following conditions:
\begin{itemize}
\item[(1)] The function $\theta:[0,1]\times\mathbb{R}^d\to\mathbb{R}^d\times\mathbb{S}_+^d\times\mathcal{L}$ is measurable, and for almost every $t\in[0,1]$ and $x\in\mathbb{R}^d$, $\theta(t,x)=(b(t,x),c(t,x),F(t,x))\in \Theta$.
\item[(2)] The continuous curve $\rho:[0,1]\to\mathcal{P}(\mathbb{R}^d)$, with $\rho_0=\mu_0$ and $\rho_1=\mu_1$, is a weak solution of the non-local FPKE \eqref{FPKE} induced by $\theta$.
\end{itemize}
Finally, by exploiting the equivalence between non-local FPKEs and martingale problems from~\cite[Corollary~1.9]{RocknerXieZhang2020PTRF}, we establish the following theorem. This result has already been proved in the continuous framework by~\cite[Theorem~4.2]{DweikGhoussoubKimPalmer2020Poincare}, where the authors used the regularity of viscosity solutions of the HJB equation.
\begin{thm}\label{PDE}
Let $\Theta\subseteq\mathbb{R}^d\times\mathbb{S}_+^d\times\mathcal{L}$ be closed, convex and satisfy Assumptions~\ref{B} and~\ref{J}. Moreover, let the cost function $\ell$ (defined in \eqref{definel}) satisfy Assumption~\ref{L}. Then
\begin{equation}
V(\mu_0,\mu_1)=PDE(\mu_0,\mu_1).
\end{equation}
\end{thm}
\begin{proof}
We first show the inequality “$\geq$”. If $V(\mu_0,\mu_1)=+\infty$ there is nothing to prove, so we assume it is finite. By Theorem~\ref{V=Vhat} there exists a probability measure $\widehat{\mathbb{P}}\in\widehat{\mathfrak{P}}_{\Theta}(\mu_0,\mu_1)$ with $J(\widehat{\mathbb{P}})=V(\mu_0,\mu_1)$. Set $\rho_t:=\widehat{\mathbb{P}}\circ X_t^{-1}$ and $\theta(t,x):=(\hat{b}(t,x),\hat{c}(t,x),\widehat{F}(t,x))$. Using~\cite[Corollary~1.9]{RocknerXieZhang2020PTRF}, $\rho$ is a weak solution of the non-local FPKE \eqref{FPKE} induced by $\theta$. Consequently $V(\mu_0,\mu_1)\geq PDE(\mu_0,\mu_1)$.
To prove the reverse inequality, assume without loss of generality that $PDE(\mu_0,\mu_1)<+\infty$ and choose any feasible pair $(\theta,\rho)$ for $PDE(\mu_0,\mu_1)$. By the superposition principle~\cite[Theorem~1.5]{RocknerXieZhang2020PTRF}, there exists a probability measure $\widehat{\mathbb{P}}\in\mathcal{P}(\Omega)$ such that
\begin{itemize}
\item[(i)] $X$ under $\widehat{\mathbb{P}}$ is a semimartingale whose differential characteristics satisfy
\[
\mathbb{E}^{\widehat{\mathbb{P}}}\!\bigl[\,(b_t^{\widehat{\mathbb{P}}},c_t^{\widehat{\mathbb{P}}},F_t^{\widehat{\mathbb{P}}})\mid X_{t-}\,\bigr]=\theta(t,X_{t-})
\qquad \widehat{\mathbb{P}}\times\mathrm{d}t\text{-a.s.}
\]
\item[(ii)] $\widehat{\mathbb{P}}\circ X_t^{-1}=\rho_t$ for every $t\in[0,1]$.
\end{itemize}
In particular $\widehat{\mathbb{P}}\in\widehat{\mathfrak{P}}_{\Theta}(\mu_0,\mu_1)$. Moreover,
\[
J(\widehat{\mathbb{P}}) = \mathbb{E}^{\widehat{\mathbb{P}}}\!\left[\int_0^1 \ell\bigl(t,X_t,\theta(t,X_t)\bigr)\,\mathrm{d}t\right]
= \int_0^1\int_{\mathbb{R}^d} \ell\bigl(t,x,\theta(t,x)\bigr)\,\rho_t(\mathrm{d}x)\,\mathrm{d}t.
\]
Taking the infimum over all such $(\theta,\rho)$ gives $V(\mu_0,\mu_1)\leq PDE(\mu_0,\mu_1)$, which completes the proof.
\end{proof}
For continuous semimartingales, we have the following analogue, which follows from~\cite[Theorem~3.8]{GuoLoeperWang2022MF}.
\begin{prop}\label{F=0PDE}
Let $\Theta\subseteq\mathbb{R}^d\times\mathbb{S}_+^d\times\{0\}$ be closed and convex, and let the cost function $\ell$ satisfy Assumptions~\ref{L} and~\ref{uL}. Then the result of Theorem~\ref{PDE} also holds.
\end{prop}

\section{Applications of the semimartingale transport theory}\label{6}
The semimartingale optimal transport framework provides a broad probabilistic generalisation of Benamou--Brenier type formulations. 
In particular, it naturally encompasses Schr\"odinger bridge problems (by taking $\Theta = \mathbb{R}^d\times\{I_d\}\times\{0\}$ with a suitable cost) and barycentric weak optimal transport (with $\Theta = \mathbb{R}^d\times\mathbb{S}_+^d\times\{0\}$); we refer to~\cite{guo2025dynamiccharacterizationbarycentricoptimal} for detailed discussions. 
In this section we illustrate its scope through three further examples: 
a new proof of the Benamou--Brenier formula based on our semimartingale optimal transport duality (Subsection~\ref{OTBB});
a PDE characterisation of martingale optimal transport under weaker regularity assumptions (Subsection~\ref{PDEMOT});
and a Kantorovich duality for martingale optimal transport with jumps on the Skorokhod $J_1$-space that substantially relaxes the conditions required in the classical literature (Subsection~\ref{MOTJ}).

\subsection{The classical optimal transport problem}\label{OTBB}
Let $\mu_0,\mu_1\in\mathcal{P}(\mathbb{R}^d)$ have finite second moments. The optimal transport problem with quadratic cost is given by
\begin{equation}
OT_2(\mu_0, \mu_1):=\inf_{\pi\in \Pi(\mu_0, \mu_1)}\int_{\mathbb{R}^d\times\mathbb{R}^d}|x - y|^2\,\pi(\mathrm{d}x, \mathrm{d}y),
\end{equation}
where $\Pi(\mu_0, \mu_1)$ denotes the set of couplings between $\mu_0$ and $\mu_1$, i.e.
\[
\pi\in\Pi(\mu_0, \mu_1)\iff\pi(A\times\mathbb{R}^d)=\mu_0(A)\ \text{and}\ \pi(\mathbb{R}^d\times A)=\mu_1(A)\qquad\forall A\in\mathcal{B}(\mathbb{R}^d).
\]
See~\cite{Villani2009OT} for an overview. In the seminal work~\cite{BB2000} it is shown that solving $OT_2(\mu_0, \mu_1)$ is equivalent to minimising the total energy along absolutely continuous curves $(\rho_t)_{t\in[0,1]}$ connecting $\mu_0$ to $\mu_1$. More precisely,
\begin{equation}\label{OT}
OT_2(\mu_0, \mu_1)=\inf_{(b, \rho)}\int_0^1\!\!\int_{\mathbb{R}^d} |b_t|^2 \,\rho_t(\mathrm{d}x)\,\mathrm{d}t,
\end{equation}
where the infimum is taken over all pairs $(b, \rho)$ such that $\rho_0=\mu_0$, $\rho_1=\mu_1$, and $(b, \rho)$ satisfies
\begin{equation}
\partial_t \rho_t + \operatorname{div}(b_t \rho_t)=0
\end{equation}
in the sense of distributions. Equation~\eqref{OT} is known as the dynamic formulation of optimal transport, or the Benamou--Brenier formula. It admits the probabilistic representation
\begin{equation}\label{OT=BB}
OT_2(\mu_0, \mu_1) = BB_2(\mu_0, \mu_1),
\end{equation}
where
\begin{equation}\label{BB}
BB_2(\mu_0, \mu_1) := \inf\biggl\{\mathbb{E}\biggl[\int_0^1 |b_t|^2\,\mathrm{d}t\biggr] \;\Big|\; \mathrm{d}X_t = b_t\,\mathrm{d}t,\ X_0 \sim \mu_0,\ X_1 \sim \mu_1 \biggr\}.
\end{equation}
Next, we directly establish an equivalence between~\eqref{BB} and the classical Kantorovich duality, without relying on the intermediate formula~\eqref{OT=BB}.
\begin{thm}\label{thm:OTdual}
Let $\mu_0,\mu_1\in\mathcal{P}(\mathbb{R}^d)$ have finite second moments. Then 
\begin{equation*}
BB_2(\mu_0,\mu_1)=\sup_{\phi\in C_b(\mathbb{R}^d)}\Bigl\{\int_{\mathbb{R}^d}\phi^c(x)\,\mu_0(\mathrm{d}x)-\int_{\mathbb{R}^d}\phi(y)\,\mu_1(\mathrm{d}y)\Bigr\},
\end{equation*}
where $\phi^c(x):=\inf_{y\in\mathbb{R}^d}\{|x-y|^2+\phi(y)\}$,
and the infimum is achieved for the problem $BB_2(\mu_0,\mu_1)$ in~\eqref{BB} whenever it is finite.
\end{thm}
\begin{proof}
Set $D(\mu_0,\mu_1):=\sup_{\phi\in C_b(\mathbb{R}^d)}\bigl\{\int_{\mathbb{R}^d}\phi^c(x)\,\mu_0(\mathrm{d}x)-\int_{\mathbb{R}^d}\phi(y)\,\mu_1(\mathrm{d}y)\bigr\}$. 
Let $\Theta:=\mathbb{R}^d\times\{0\}\times\{0\}$ and define $L(t,\omega,b,c,F)=|b|^2$. By Proposition~\ref{F=0V=Vhat} and the definition of~\eqref{SOT} we obtain
\begin{align*}
V(\mu_0,\mu_1)=\widehat{V}(\mu_0,\mu_1)
&=\inf_{\mathbb{P}\in\widehat{\mathfrak{P}}_{\Theta}(\mu_0,\mu_1)}\mathbb{E}^{\mathbb{P}}\biggl[\int_{0}^{1} |b_t^{\mathbb{P}}|^2\,\mathrm{d}t\biggr]\\
&=\inf\biggl\{\mathbb{E}^{\mathbb{P}}\biggl[\int_{0}^{1} |b_t^{\mathbb{P}}|^2\,\mathrm{d}t\biggr]\ \bigg|\ \mathrm{d}X_t=b_{t}^{\mathbb{P}}\,\mathrm{d}t,\ \mathbb{P}\circ X_0^{-1}=\mu_0,\ \mathbb{P}\circ X_1^{-1}=\mu_1\biggr\}\\
&=BB_2(\mu_0,\mu_1).
\end{align*}
By Proposition~\ref{F=0SOT} the duality holds:
\begin{equation*}
V(\mu_0,\mu_1)=\mathcal{V}(\mu_0,\mu_1)=\sup_{\lambda_1\in C_b(\mathbb{R}^d)}\Bigl\{ \int_{\mathbb{R}^d}\lambda_0^{\lambda _1}(x)\,\mu_0(\mathrm{d}x)-\int_{\mathbb{R}^d}\lambda_1(x)\,\mu_1(\mathrm{d}x)\Bigr\},
\end{equation*}
and the infimum for $BB_2(\mu_0,\mu_1)$ of~\eqref{BB} is attained whenever it is finite.
For any $x,y\in\mathbb{R}^d$ and any $\mathbb{P}\in\mathfrak{P}_{\Theta}(\delta_x,\delta_y)$, the identity $X_1-X_0 = \int_0^1 b_t^{\mathbb{P}}\mathrm{d}t$ holds $\mathbb{P}$-a.s., whence $|y-x|^2 = \bigl|\mathbb{E}^{\mathbb{P}}[X_1-X_0]\bigr|^2 \le \mathbb{E}^{\mathbb{P}}[|X_1-X_0|^2]$. Using the definition of $\lambda_0^{\lambda_1}$ and choosing the constant drift $b_t^{\mathbb{P}^{\ast}}\equiv y-x$, we obtain
\begin{align*}
\lambda_0^{\lambda_1}(x)
&= \inf_{\mathbb{P}\in\mathfrak{P}_{\Theta}(\delta_x)}\mathbb{E}^{\mathbb{P}}\biggl[\int_{0}^{1}|b_{t}^{\mathbb{P}}|^2\,\mathrm{d}t+\lambda_1(X_1)\biggr] \\
&\leq \inf_{y\in\mathbb{R}^d}\inf_{\mathbb{P}\in\mathfrak{P}_{\Theta}(\delta_x,\delta_y)}\mathbb{E}^{\mathbb{P}}\biggl[\int_{0}^{1}|b_{t}^{\mathbb{P}}|^2\,\mathrm{d}t+\lambda_1(y)\biggr] \\
&\leq \inf_{y\in\mathbb{R}^d}\mathbb{E}^{\mathbb{P}^{\ast}}\biggl[\int_{0}^{1}|y-x|^2\,\mathrm{d}t+\lambda_1(y)\biggr] \\
&= \inf_{y\in\mathbb{R}^d}\bigl\{|y-x|^2+\lambda_1(y)\bigr\}
= \lambda_1^{c}(x),
\end{align*} 
where $\lambda_1^c(x) := \inf_{y\in\mathbb{R}^d} \{|x-y|^2 + \lambda_1(y)\}$. Hence $V(\mu_0,\mu_1)=\mathcal{V}(\mu_0,\mu_1)\le D(\mu_0,\mu_1)$.
To prove the reverse inequality, note that for any $\phi\in C_b(\mathbb{R}^d)$ we have $\phi^c(x)-\phi(y)\le|x-y|^2$. For any $\mathbb{P}\in\mathfrak{P}_{\Theta}(\mu_0,\mu_1)$, it follows that
\begin{align*}
\int_{\mathbb{R}^d}\phi^c(x)\,\mu_0(\mathrm{d}x)-\int_{\mathbb{R}^d}\phi(y)\,\mu_1(\mathrm{d}y)
&=\mathbb{E}^{\mathbb{P}}[\phi^c(X_0)-\phi(X_1)]\\
&\le \mathbb{E}^{\mathbb{P}}[|X_1-X_0|^2]\\
&=\mathbb{E}^{\mathbb{P}}\biggl[\biggl|\int_0^1 b_t^{\mathbb{P}}\,\mathrm{d}t\biggr|^2\biggr]
\le \mathbb{E}^{\mathbb{P}}\biggl[\int_0^1 |b_t^{\mathbb{P}}|^2 \,\mathrm{d}t\biggr],
\end{align*}
where the last step uses Jensen's inequality. Taking the supremum over $\phi$ and the infimum over $\mathbb{P}$ gives $D(\mu_0,\mu_1)\le V(\mu_0,\mu_1)$, which completes the proof.
\end{proof}
\begin{rem}
Combining Theorem~\ref{thm:OTdual} with the classical Kantorovich duality~\cite[Theorem 5.10]{Villani2009OT}, we obtain a new proof of the Benamou--Brenier formula~\eqref{OT=BB}.
\end{rem}

\subsection{The PDE characterisation of the martingale optimal transport problem}\label{PDEMOT}
The Benamou--Brenier formulation of martingale optimal transport has been established in~\cite{HuesmannTrevisan2019Bernoulli} and~\cite{BackhoffVeraguasBeiglbock2020AOP}; we adopt the former, as it fits more naturally with our framework. Let $\Theta:=\{0\}\times\mathbb{S}_+^{d}\times\{0\}$. For $\mu_0,\mu_1\in\mathcal{P}(\mathbb{R}^d)$ with finite second moments, the martingale optimal transport problem with cost $\ell:[0,1]\times\mathbb{R}^d\times\mathbb{S}_+^{d}\to\mathbb{R}^{+}$ is
\begin{equation}\label{MBB}
MBB(\mu_0, \mu_1):=\inf_{\mathbb{P}\in\mathfrak{P}_{\Theta}(\mu_0, \mu_1)}\mathbb{E}^{\mathbb{P}}\biggl[\int_{0}^{1}\ell(t,X_t,c_t^{\mathbb{P}})\,\mathrm{d}t\biggr].
\end{equation}
By Proposition~\ref{F=0SOT} we immediately obtain the following duality theorem.
\begin{thm}
Let $\mu_0,\mu_1\in\mathcal{P}(\mathbb{R}^d)$ have finite second moments and let $\ell$ (defined in \eqref{definel}) satisfy Assumptions~\ref{L} and~\ref{uL}. Then 
\begin{equation*}
MBB(\mu_0,\mu_1)=\sup_{\lambda_1\in C_b(\mathbb{R}^d)}\Bigl\{ \int_{\mathbb{R}^d}\lambda_0^{\lambda _1}(x)\,\mu_0(\mathrm{d}x)-\int_{\mathbb{R}^d}\lambda_1(x)\,\mu_1(\mathrm{d}x) \Bigr\},
\end{equation*}
where
\begin{equation*}
\lambda_{0}^{\lambda_1}(x)=\inf_{\mathbb{P}\in\mathfrak{P}_{\Theta}(\delta_x)}\mathbb{E}^{\mathbb{P}}\biggl[\int_0^1 \ell(t,X_t,c_t^{\mathbb{P}})\,\mathrm{d}t+\lambda_1(X_1)\biggr].
\end{equation*}
\end{thm}
We now give an equivalent formulation in terms of the FPK equation. Set
\begin{equation}\label{MPDE}
MPDE(\mu_0, \mu_1):=\inf_{(c, \rho)}\int_0^1\!\!\int_{\mathbb{R}^d} \ell(t,x,c_t(x)) \,\rho_t(\mathrm{d}x)\,\mathrm{d}t,
\end{equation}
where the infimum is taken over all pairs $(c, \rho)$ such that $\rho_0=\mu_0$, $\rho_1=\mu_1$, and $(c, \rho)$ solves
\begin{equation}
\partial_t\rho_t=\frac{1}{2} D_{xx}^{2}(c_t\cdot\rho_t)
\end{equation}
in the sense of distributions.  

By Proposition~\ref{F=0SOT} and~\ref{F=0PDE}, we obtain the following generalisation of~\cite[Theorem~3.3]{HuesmannTrevisan2019Bernoulli}, in which the $p$-growth condition on the cost function is removed.
\begin{thm}
Let $\mu_0,\mu_1\in\mathcal{P}(\mathbb{R}^d)$ have finite second moments and let $\ell$ (defined in \eqref{definel}) satisfy Assumptions~\ref{L} and~\ref{uL}. Then 
\begin{equation*}
MBB(\mu_0,\mu_1)=MPDE(\mu_0,\mu_1),
\end{equation*}
and the infimum for the problem $MBB(\mu_0,\mu_1)$ of~\eqref{MBB} is achieved whenever it is finite.
\end{thm}

\subsection{The martingale optimal transport problem on the Skorokhod \texorpdfstring{$J_1$}{J1}-space}\label{MOTJ}
Martingale optimal transport on the Skorokhod space was introduced by Dolinsky and Soner~\cite{DS2015SPA}, who proved a Kantorovich-type duality for path-dependent payoffs under the $J_1$-topology, albeit under strong regularity assumptions on the cost (uniform $J_1$-continuity and linear growth). Subsequent works~\cite{GuoTanTouzi2017SPA, CKPS2021MA} have relaxed these conditions by employing alternative topologies, but at the expense of working in non-metrisable settings that complicate numerical discretisation and constructive arguments.  In contrast, the Polish $J_1$-topology supports explicit time-discretisation schemes; a distinctive feature of our work is that the results are obtained without imposing the usual conditions on the filtered probability space, thereby avoiding delicate measurability technicalities.
We first define
\begin{equation}
\mathfrak{P}_{\Theta}^M(\mu_0,\mu_1)
:=\bigl\{\mathbb{P}\in\mathfrak{P}_{\Theta}(\mu_0,\mu_1)\mid
X\text{ is a }\mathbb{P}\text{-}\mathbb{F}\text{-martingale}\bigr\}.
\end{equation}
To obtain a precise characterisation, in terms of semimartingale characteristics, of when a semimartingale is a true martingale, we impose the following condition on $\Theta$.
\begin{ass}\label{Theta_M}
A set $\Theta \subseteq \mathbb{R}^d \times \mathbb{S}_+^d \times \mathcal{L}$
satisfies Assumption~\ref{Theta_M} if 
\[
b+\int_{\mathbb{R}^d} [z-h(z)]\,F(\mathrm{d}z)=0,\qquad\forall(b,c,F)\in\Theta.
\]
\end{ass}
\begin{rem}
This condition does not depend on the choice of the truncation function $h$ (see \cite[Proposition~2.24]{JacodShiryaev2003}).
\end{rem}
Consequently, we have the following property from~\cite[Remark~2.3]{LiuNeufeld2019TAMS}.
\begin{prop}\label{V=V^M}
If $\mu_0$ has finite first moment and $\Theta$ satisfies Assumptions~\ref{B} and~\ref{Theta_M}, then
\begin{equation}
\mathfrak{P}_{\Theta}^M(\mu_0,\mu_1)=\mathfrak{P}_{\Theta}(\mu_0,\mu_1).
\end{equation}
\end{prop}
For $\mu_0,\mu_1\in\mathcal{P}(\mathbb{R}^d)$ with finite first moments, set
\begin{align}
V^M(\mu_0,\mu_1) &:= \inf_{\mathbb{P}\in\mathfrak{P}_{\Theta}^M(\mu_0,\mu_1)} J(\mathbb{P}),\\
\mathcal{V}^M(\mu_0,\mu_1) &:=\sup_{\lambda_1\in C_b(\mathbb{R}^d)}
\Bigl\{ \int_{\mathbb{R}^d} \lambda_0^{\lambda_1}(x)\,\mu_0(\mathrm{d}x)
      -\int_{\mathbb{R}^d}\lambda_1(x)\,\mu_1(\mathrm{d}x) \Bigr\},
\end{align}
where $\lambda_0^{\lambda_1}$ is defined as in~\eqref{Rc} with $\mathfrak{P}_{\Theta}^M(\delta_x)$ in place of $\mathfrak{P}_{\Theta}(\delta_x)$. Under Assumptions~\ref{B},~\ref{J} and~\ref{Theta_M}, Proposition~\ref{V=V^M} yields $\mathfrak{P}_{\Theta}^M(\mu_0,\mu_1)= \mathfrak{P}_{\Theta}(\mu_0,\mu_1)$; consequently $V^M = V$ and $\mathcal{V}^M = \mathcal{V}$.  Theorem~\ref{D} therefore immediately implies the following duality.
\begin{cor}
Assume that $\mu_0,\mu_1\in\mathcal{P}(\mathbb{R}^d)$ have finite first moments. Let $\Theta\subseteq\mathbb{R}^d\times\mathbb{S}_+^d\times\mathcal{L}$ be closed, convex, and satisfy Assumptions~\ref{B}, \ref{J} and~\ref{Theta_M}. Moreover, let the cost function $L$ satisfy Assumption~\ref{L}. Then
\begin{equation}
V^M(\mu_0,\mu_1)=\mathcal{V}^M(\mu_0,\mu_1).
\end{equation}
\end{cor}
Finally, we specialise the above to a setting that substantially weakens the assumptions of~\cite[Theorem~2.9]{DS2015SPA}.
\begin{cor}
Assume $\mu_0=\delta_1$ and that $\mu_1\in\mathcal{P}(\mathbb{R}^d)$ has finite first moments.  Set $L(t,\omega,b,c,F)=G(\omega)$, where $G:\Omega\to[0,+\infty)$ is lower semicontinuous. Let
$\Theta\subseteq\mathbb{R}^d\times\mathbb{S}_+^d\times\mathcal{L}$ be closed, convex, and satisfy Assumptions~\ref{B}, \ref{J} and~\ref{Theta_M}. Then
\begin{equation}
V^M(\delta_1, \mu_1)=\mathcal{V}^M(\delta_1, \mu_1).
\end{equation}
\end{cor}
\begin{rem}
Using the Snell envelope approach, together with the Doob--Meyer decomposition, one can obtain an equivalent representation of $\mathcal{V}^M(\delta_1, \mu_1)$ that precisely recovers Theorem~2.9 in~\cite{DS2015SPA}. For more details, see Theorem~2.4 in~\cite{GuoTanTouzi2016SIAMJCO}.
\end{rem}

\end{document}